\documentclass[12pt, a4paper]{article}

\usepackage[utf8]{inputenc}
\usepackage[english]{babel}
\usepackage{amsmath}
\usepackage{multicol}
\usepackage{graphicx}
\usepackage[colorinlistoftodos]{todonotes}
\usepackage[margin=1.3in]{geometry}
\usepackage{graphicx}
\usepackage{amsfonts}
\usepackage{enumitem}
\usepackage{tikz}
\usepackage{tikz-cd}
\usetikzlibrary{decorations.markings,positioning}
\usepackage{amsthm}
\usepackage{amssymb}
\usepackage{cite}
\usepackage{float}
\usepackage{microtype}
\usepackage{mathtools}
\usepackage{hyperref}

\newcommand{\R}{\mathbb{R}}

\newcommand{\p}{\mathbb{P}}

\newcommand{\B}{\mathcal B}
\newcommand{\T}{\mathcal T}
\newcommand{\tmd}{\mathrm{TMD}}
\newcommand{\tns}{\mathrm{TNS}}
\newcommand{\trn}{\mathrm{TRN}}
\newcommand{\bsim}{\underset{\tiny{\mathrm{bar}}}{\sim}}
\newcommand{\csim}{\underset{\tiny{\mathrm{comb}}}{\sim}}
\newcommand{\tsim}{\underset{\tiny{\mathrm{tmd}}}{\sim}}

\newtheorem{theorem}{Theorem}[section]

\newtheorem{lemma}[theorem]{Lemma}
\newtheorem{proposition}[theorem]{Proposition}

\theoremstyle{definition}

\newtheorem{remark}[theorem]{Remark}
\newtheorem{example}[theorem]{Example}
\title{From trees to barcodes and back again: theoretical and statistical perspectives}

\author{Lida Kanari$^{1*}$, Adélie Garin$^2$, Kathryn Hess$^2$}

\date{%
    $^{1}$ Blue Brain Project, École polytechnique fédérale de Lausanne (EPFL), Campus Biotech, 1202 Geneva, Switzerland.\\
    $^{2}$ Laboratory for topology and neuroscience, Brain Mind Institute, École polytechnique fédérale de Lausanne (EPFL), 1015 Lausanne, Switzerland.\\[2ex]%
    *Correspondence: lida.kanari@gmail.com \\
    \today
}

\begin{document}
\maketitle

\abstract{Methods of topological data analysis have been successfully applied in a wide range of fields to provide useful summaries of the structure of complex data sets in terms of topological descriptors, such as persistence diagrams. While there are many powerful techniques for computing topological descriptors, the inverse problem, i.e., recovering the input data from  topological descriptors, has proved to be challenging. In this article we study in detail the Topological Morphology Descriptor (TMD), which assigns a persistence diagram to any tree embedded in Euclidean space, and a sort of stochastic inverse to the TMD, the Topological Neuron Synthesis (TNS) algorithm, gaining both theoretical and computational insights into the relation between the two. We propose a new approach to classify barcodes using symmetric groups, which provides a concrete language to formulate our results. We investigate to what extent the TNS recovers a geometric tree from its TMD and describe the effect of different types of noise on the process of tree generation from persistence diagrams. We prove moreover that the TNS algorithm is stable with respect to specific types of noise.}

\bigskip


\clearpage

\tableofcontents

\clearpage

\section{Introduction}

Although geometric approaches to analyzing data have been extensively used for many years, the first topological methods for data analysis were developed only recently, e.g., ~\cite{Frosini1997},~\cite{Edelsbrunner2002},~\cite{Robins2002},~\cite{Zomorodian2004}, ~\cite{Verri2004} and~\cite{Carlsson2009}. Topological Data Analysis (TDA) is a fairly new field at the intersection of data science and algebraic topology, the aim of which is to provide robust mathematical, statistical, and algorithmic methods to infer and analyze the topological and geometric structures underlying complex data. These data are often represented as point clouds in Euclidean or metric spaces, though TDA methods have also been generalized to geometric objects and graphs. TDA has proved its utility in a wide range of applications in biology~\cite{Harrington2012},~\cite{Byrne2019},~\cite{Martino2018},~\cite{Gameiro2015}, material science~\cite{Lee2018}, and climate science~\cite{Muszynski2019}, among other fields. Although it is still rapidly evolving, TDA now provides a set of powerful and efficient tools that can be used in combination with or as complements to other data science tools.

One of the most promising applications of TDA is to the study of the brain, where it has served to analyze neuronal morphologies~\cite{tmd}, brain networks~\cite{Reimann2017},~\cite{Sizemore2017}~\cite{tmd}, and brain functionality~\cite{Stolz2017}. Motivated by the desire to objectively classify neuronal morphologies, in a previous publication (Kanari and Hess in~\cite{tmd}) we designed a topological signature for trees, the Topological Morphology Descriptor (TMD), that assigns a \emph{barcode} (i.e., a multi-set of open intervals -- called \emph{bars} -- in the real line) to any \emph{geometric tree} (i.e, any finite binary tree embedded in $\mathbb R^3$). We showed that the TMD algorithm effectively determines the reliability of the clustering of random and neuronal trees. Moreover, using the TMD algorithm, we performed an objective, stable classification of pyramidal cells in the rat neocortex~\cite{objective_classification}, based only on the shape of their dendrites.

A frequent topic of discussion in the context of TDA is how to define an inverse to the process of associating a particular topological descriptor to a dataset, i.e., how to design a practical algorithm to recover the input data from a topological descriptor, such as a barcode. Oudot and Solomon~\cite{Oudot2018} and Curry et al.~\cite{Curry2018} have proposed partial solutions to this problem. The main obstacle that renders this endeavor particularly challenging has proven to be the computational complexity of the  space of inputs considered. To avoid this obstacle, it is reasonable to constrain the input space and search only for an inverse transformation that is relevant in a specific context, for instance to look for solutions only in the space of embedded graphs, as in~\cite{Belton2020}.

In the context of geometric trees, we have designed an algorithm to reverse-engineer the TMD~\cite{tns}, in order to digitally generate artificial neurons, to compensate for the dearth of available biological reconstructions. This algorithm, called Topological Neuron Synthesis (TNS), stochastically generates a geometric tree from a barcode, in a biologically grounded manner. As shown in~\cite{tns}, the synthesized neurons are statistically indistinguishable from the corresponding reconstructed neurons in terms of both their morphological characteristics and the networks they form.

In this article, we further study the properties of this generative algorithm, from mathematical and statistical perspectives. We perform a theoretical and computational analysis of the TMD and TNS algorithms and their mathematical properties, in which symmetric groups play a key role. In particular, we investigate in detail the extent to which the TNS provides an inverse to the TMD.

First, we carefully define our objects of study -- geometric trees, barcodes, and persistence diagrams -- then recall the TMD and TNS algorithms.  We also introduce two distinct classifications of geometric trees: into combinatorial types and into TMD-types.  The symmetric groups play an important role in our classification of trees into TMD-types. These complementary descriptions provide us with a language in which to formulate our results on the relationship between the TMD and the TNS.

In the next section, we introduce tools to describe the set of geometric trees that realize a specific barcode, i.e., whose TMD is equal to that barcode. In particular we establish an explicit formula for the cardinality of this set, which we use to describe how the cardinality changes when a new bar is added to a barcode or two bars of a barcode permuted. Cayley graphs of symmetric groups provide a useful visualization of these effects.

We then study the composite of the TNS and TMD algorithms from a theoretical perspective, to quantify the extent to which the TNS acts as an inverse to the TMD. For a given barcode $B$, we show that, for a reasonable choice of parameter in the TNS, the probability that the bottleneck distance between the barcodes $B$ and $\tmd \circ \tns (B)$ is greater than $\varepsilon$ decreases with $\varepsilon$, thus establishing a form of stability for the TNS. We prove, moreover, that the probability that two bars of a barcode $B$ will be permuted by applying $\tmd \circ \tns$ decreases exponentially with the distance between the terminations of the two bars, which is another form of stability. Together these stability results imply that the TNS is an excellent approximation to a (right) inverse to the TMD.

In the final section we present computational results that illustrate the complex relationship between a barcode and its possible tree-realizations. In particular, we study the distinguishing characteristics of ``biological'' geometric trees, i.e., those that arise from digital reconstructions of neurons, as opposed to arbitrary geometric trees. We also show that both the combinatorial type and the TMD-type of a geometric tree can change significantly when applying the composite $\tns \circ \tmd$, from which it follows that the TNS is not a left inverse to the TMD.

\begin{figure}[H]
    \centering
    \includegraphics[scale=0.45]{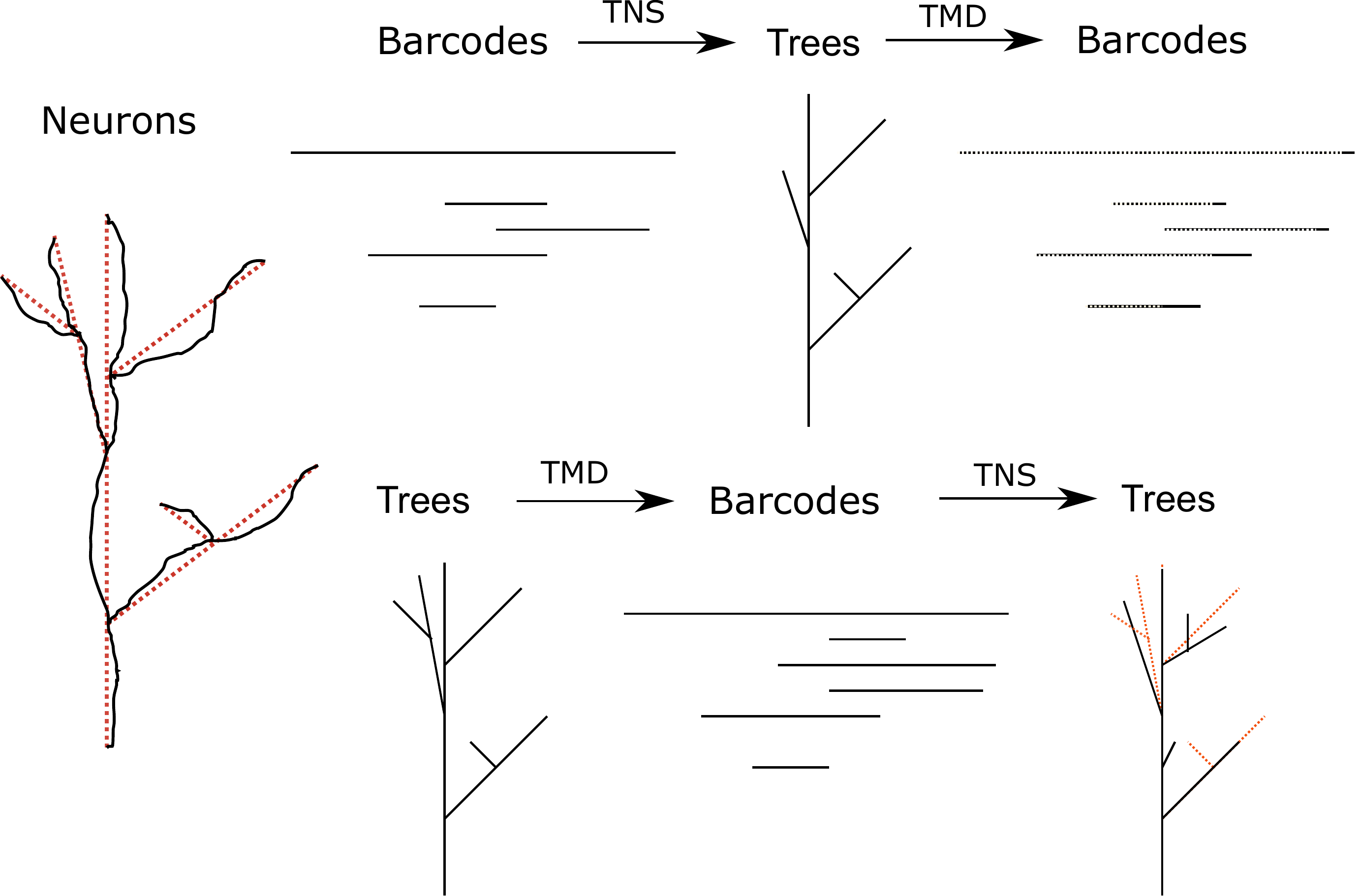}
    \caption{The two composites of TMD and TNS. (Left) An illustration of how a neuron (black) is modeled as a tree (dashed red lines). We describe this process and how to extract a barcode from a tree in section \ref{sec:tmd}. (Top) The composite $\tmd \circ \tns$ applied to a barcode $B$. The new barcode $B' = \tmd \circ \tns (B)$ is indicated in dashed lines on top of the barcode $B$ on the right. We show in section \ref{sec:stability} that the barcodes $B$ and $B'$ will almost certainly be very similar and quantify this similarity.
    (Bottom) The composite $\tns \circ \tmd$ applied to a tree $T$. The tree $T$ that we start with is indicated in dashed red lines under the new tree $T' = \tns \circ \tmd (T)$. The trees $T$ and $T'$ can be quite different combinatorially, as seen on the right.}
    \label{intro_trees_barcodes}
\end{figure}


\section{Mathematical background}

Precisely what a mathematician means by the terms ``tree'' and ``barcode'' can vary depending on context. First, we specify what these terms mean in this article. We then recall biologically motivated algorithms for generating barcodes from trees \cite{tmd} and trees from barcodes \cite{tns}, the relation between which will be made clear in the following sections.

\subsection{Trees}\label{sec:trees}

A \emph{finite rooted tree} $T$ is an acyclic, finite, directed graph such that each vertex is of degree at most 3, with a distinguished vertex $r$ of degree 1, called the \emph{root}. A vertex $v$ of $T$ is a \emph{parent} of a vertex $w$ if there is a directed edge from $w$ to $v$; the vertex $w$ is then a \emph{child} of $v$. Each vertex of $T$ has a single parent, except for the root $r$, which has no parent, and at most two children. The non-root vertices of degree $1$ are called the \emph{leaves} of $T$, and the vertices of degree $3$ the \emph{branch points} of $T$. A finite tree $T$ is fully specified by its set of vertices, equipped with the partial order ``is a parent of''.

Our main objects of study in this article are \emph{geometric trees}, i.e., embeddings of finite rooted trees in $\R^3$, which are often used to model neurons. We assume, moreover, that if a vertex $v$ is the parent of  a vertex $w$, then the distance from the root to $v$ is less than that from the root to $w$.\footnote{Much of the framework that we we develop in this article could be extended to geometric trees equipped with a different distance function that does not necessarily satisfy this condition, at the price of a more involved combinatorial representation. Since our geometric trees of interest -- digitally reconstructed neurons equipped with the path-distance from the root -- do satisfy our extra assumption, we prefer to leave this extension to the interested reader.} Let $\T$ denote the set of geometric trees.  

We say that two geometric trees $T$ and $T'$ are \emph{combinatorially equivalent}, denoted $T\csim T'$, if they are embeddings of the same finite rooted tree. In other words, the combinatorial type of a geometric tree is independent of its embedding in $\R^3$.  

\subsection{Barcodes} \label{sec:barcodes}

A persistence \emph{barcode} is a finite multi-set $B = \{ (b_i,d_i) \}_{i =0,...,n}$ of open intervals, called \emph{bars}, in the real line, $\R$. We call $b_i$ the \emph{birth time} and $d_i$ the \emph{death time} of the $i^{\mathrm{th}}$ interval. For barcodes of geometric trees, generated with the $\tmd$ algorithm (cf.~section \ref{sec:tmd}), the birth time $b_i$ is the distance from the first bifurcation of branch $i$ to the root, while the death time $d_i$ is the distance from the branch termination to the root. Let $\B$ denote the set of all barcodes.

A persistence barcode can equivalently be represented as a multi-set of points in $\R^2$, called a \emph{persistence diagram}, where a bar $(b_i,d_i)$ corresponds to a point in $\R^2$ with $x$-coordinate $d_i$ and $y$-coordinate $b_i$.\footnote{The inversion of the coordinates that we apply here is motivated by the TMD algorithm (cf.~section \ref{sec:tmd}), which decomposes a geometric tree into bars, starting at the leaves and progressing down to the root.} If $B$ is a barcode, we let $\mathrm{PD}(B)$ denote the associated persistence diagram. Note that, under this correspondence, the  points of $\mathrm{PD}(B)$ lie below the diagonal, since $b_i$ is less than $d_i$ for every $i$.

We say that a barcode $B = \{ (b_i,d_i) \}_{i =0,...,n}$ is \emph{strict} if the first bar $(b_0,d_0)$ properly contains all the others, i.e.,  $b_0 < b_i$ and $d_i < d_0$ for all $i$, and no bars are born or die at the same time, i.e., $b_i \neq b_j$ and $d_i \neq d_j$  for all $i \neq j$. The birth times of a strict barcode admit a total ordering.
Without loss of generality, we assume that the bars are ordered by birth value, that is $b_0 < b_1 < \dots < b_n$.  Let $\B^{\mathrm{st}}$ denote the set of strict barcodes, and let $\B^{\mathrm{st}}_n$ denote the subset of those strict barcodes with $n+1$ bars. 

We say that two strict barcodes $B =\{ (b_i,d_i) \}_{i =0,...,n}$ and $B' =\{ (b'_i,d'_i) \}_{i =0,...,n}$ with the same number of bars \emph{are equivalent}, denoted $B \bsim B'$, if their deaths occur in the same order ($d_{i} >d_j\Leftrightarrow d'_{i} > d'_j$), which clearly defines an equivalence relation on $\B^{\mathrm{st}}_n$. There is a bijection from the set of equivalence classes of strict barcodes with $n+1$ bars to the symmetric group $\mathfrak S_n$ on $n$ letters, taking the equivalence class of  a barcode $B =\{ (b_i,d_i) \}_{i =0,...,n}$ to the permutation $\sigma_B: \{1,...,n\} \to \{1,...,n\}$, where $\sigma_B(i)=\# \{ j \mid d_j \leq d_i\}.$

We denote the equivalence class containing a strict barcode $B = \{ (b_i,d_i) \}_{i =0,...,n} $ by $({i_1}...{i_n})$, where $d_{i_k} > d_{i_{k+1}}$ for all $1 \leq  k <  n$. For example, $(2134)$ corresponds to the barcode with $5$ bars shown in Figure~\ref{fig:barcode_2134}.

\begin{figure}[H]
\centering
\includegraphics[scale=0.5]{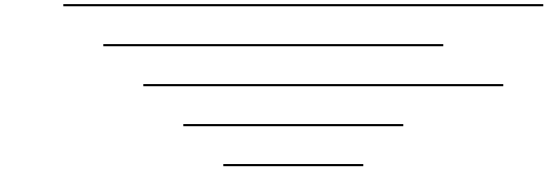} 
\put(-195,60){$b_0$}
\put(5,60){$d_0$}
\put(-180,45){$b_1$}
\put(-35,45){$d_1$}
\put(-165,30){$b_2$}
\put(-10,30){$d_2$}
\put(-150,15){$b_3$}
\put(-50,15){$d_3$}
\put(-135,0){$b_4$}
\put(-65,0){$d_4$}
\label{barcode_2134}
\caption{A strict barcode belonging to the equivalence class $(2134)$. One bar $(b_0,d_0)$ contains all the others. The remaining bars are ordered by their birth times $(b_1 < b_2 < b_3 < b_4)$. Similarly, the deaths are ordered $d_2 > d_1 > d_3 >d_4$, leading to the notation $(2134)$.}
\label{fig:barcode_2134}
\end{figure}
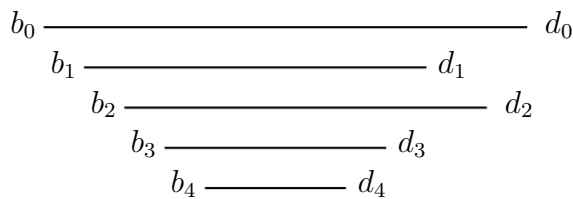

\subsection{The TMD: from trees to barcodes}\label{sec:tmd}

The TMD (Topological Morphology Descriptor) is a many-to-one function from the set of geometric trees to the set of barcodes,
$$\tmd: \T \to \B,$$
that encodes the overall shape of the tree, both the topology of the branching structure of a tree and its embedding in $\R^3$ \cite{tmd}.  It is defined recursively as follows.

Let $T$ be a rooted tree with root $r$ and set $N$ of vertices, with subset $L$ of leaves. Let $\delta: N\to \R_{\geq 0}$ be the function that assigns to each vertex its Euclidean distance to the root $r$.

Intuitively, the output of the TMD algorithm is a barcode,\footnote{ For those used to think in terms of persistent homology, the TMD computes the $0$-dimensional barcode, or persistence diagram, of the distance function $\delta$. Each bar $(b,d)$ corresponds to a connected component in the sublevel sets $\delta ^{-1}\big([0, t)\big)$, that is, a branch of the tree. Note that the birth and death roles are reversed in the TMD algorithm compared to "usual" terminology in persistent homology: the birth corresponds to the bifurcation and the death to the termination of a branch.} where each bar represents a branch of the tree. The endpoints of a bar correspond to the distances to the root from the tip of the branch  and  from the point where the branch bifurcates from another, longer branch, see Figure~\ref{tmd_algo}. 
Table~\ref{table_tmd_tns} summarizes the terminology used for the TMD algorithm.

\begin{figure}[H]
    \centering
    \includegraphics[scale=0.35]{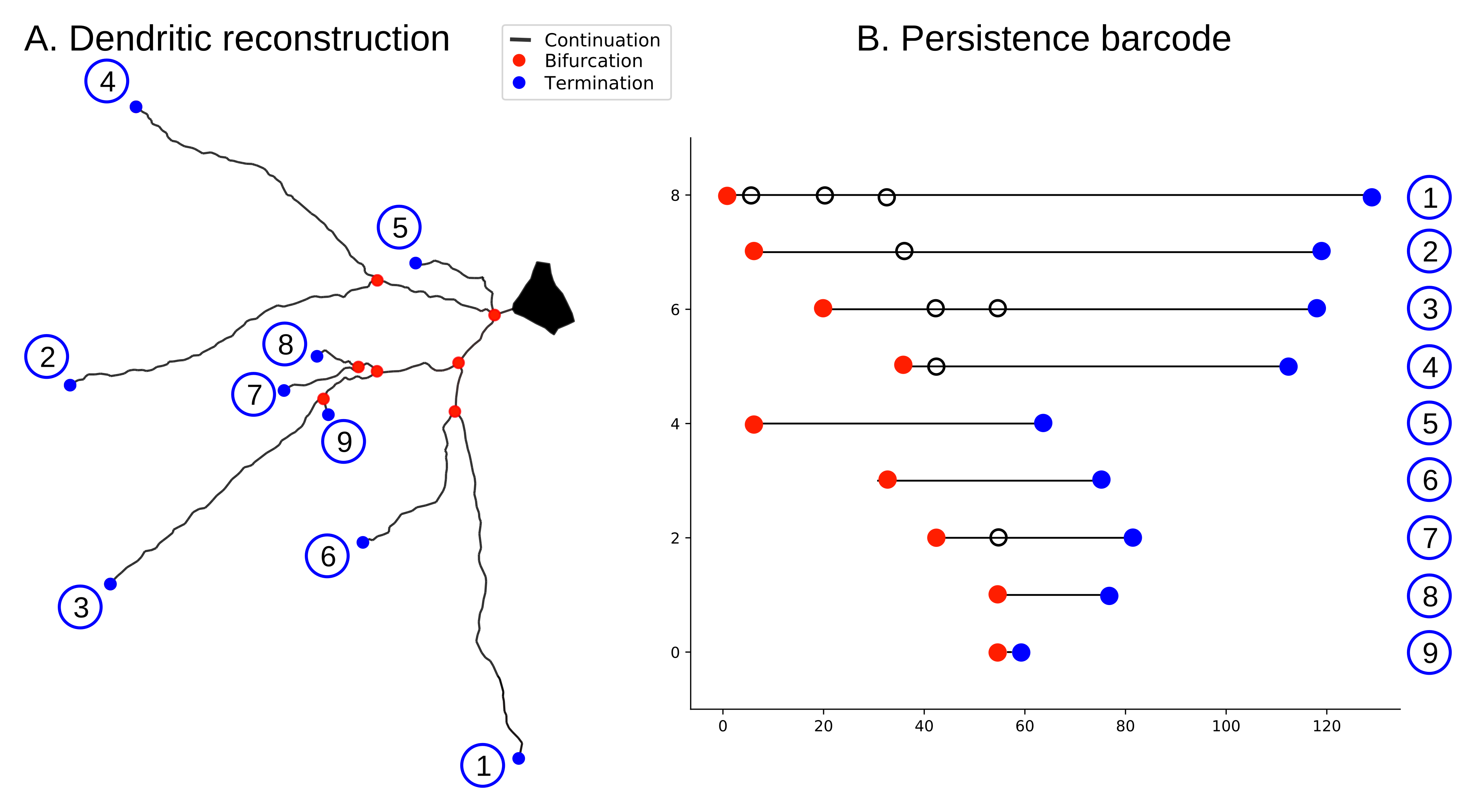}
    \caption{The algorithm to encode a tree structure as a persistence barcode. A. Neuronal tree. B. Persistence barcode generated with TMD. Each branch in the tree (A) corresponds to a bar in the barcode (B); the circled numbers encode the correspondence between branches and bars. Terminations are shown in blue, bifurcations in red, and branches in between in black.}
    \label{tmd_algo}
\end{figure}

For each $v\in N\smallsetminus L$, let $L_v$ denote the set of leaves of the subtree of $T$ with root at the branch point $v$. Let $\mu \colon N \to \mathbb{R}$ be the function defined by 
$$\mu(v) = \begin{cases}\mathrm{max}\big\{ \delta(l) \,|\, l\in L_v \big\}&: v\in N\smallsetminus L,\\ \delta(v) &: v\in L.\end{cases}$$ 
We order the children of any vertex of $T$ by their $\mu$-value: if $v_1, v_2 \in N$ are siblings, then $v_1$ is younger than $v_2$ if $\mu(v_{1})<\mu(v_{2})$.

The algorithm that extracts the TMD of a geometric tree $T$ proceeds as follows (Figure~\ref{tmd_algo}).  Start by creating a set $A$ of \emph{active vertices}, originally set equal to $L$, and an empty barcode. For each leaf $l$, the algorithm proceeds recursively along its unique path to the root $r$. At each branch point $b$, one applies the standard \emph{Elder Rule} from topological data analysis~\cite{curry-2019}, removing from $A$ all of the children of $b$, and adding $b$ to $A$. One bar is added to the barcode for each child of $b$ except (any one of) the longest. Each child removed from $A$ corresponds to a path from some leaf $l$ to $b$, which is recorded in the barcode as a bar $\big(\delta (b), \delta(l)\big)$.
These operations are applied iteratively to all the vertices until the root $r$ is reached, at which point $A$ contains only $r$ and a leaf $l$ for which $\mu$ is maximal among all leaves, which is recorded in the barcode as a bar $\big(0, \delta (l)\big)$.

If $T$ is a digital reconstruction of a neuron, and the function $\delta$ is the path distance from the soma, then $\tmd(T)$ is actually a strict barcode. Indeed, the probability for two branch points or leaves to be exactly the same distance from the soma is almost zero, and $\tmd(T)$ always has a longest bar that contains all the others.  This observation justifies our interest in the subset of strict barcodes.

The TMD gives rise to an equivalence relation on $\T$:  two geometric trees $T$ and $T'$ are \emph{TMD-equivalent}, denoted $T\tsim T'$,  if $\tmd(T) \bsim \tmd(T')$. We provide below an in-depth analysis of the TMD-equivalence classes of geometric trees. Geometric trees can be combinatorially equivalent without being TMD-equivalent and vice-versa, cf. Figure~\ref{treerez}.

\subsection{The TNS: from barcodes to trees}\label{sec:tns}

The topological neuron synthesis (TNS) algorithm \cite{tns} stochastically generates synthetic neurons, in particular for use in digital reconstuctions of brain circuitry \cite{Markram2015}.  In this paper, we focus on the sub-process of the TNS that stochastically generates a geometric tree from a strict barcode, in such a way that if a tree $T$ is generated from a barcode $B$, then $\tmd(T)$ is ``close to'' to $B$, with respect to an appropriate metric on the set of barcodes, up to some stochastic noise, cf. section \ref{sec:stability}.  Henceforth, when we refer to the TNS, we mean this sub-process.

To grow geometric trees, the TNS algorithm first initiates growth, then loops through steps of \emph{elongation} and \emph{branching/termination}. Each branch of the tree is elongated as a directed random walk~\cite{Aslangul1993} with memory. At each step, a growing tip is assigned probabilities to bifurcate, to terminate, or to continue that depend on the path distance from the root and on a chosen bar of the selected barcode. Once a bar has been used, it is removed from the barcode. The growth of a tree terminates when no bars remain to be used.  We now provide further details of the two steps in this process.

\subsubsection*{Bifurcation / Termination} \label{sec:bifurcation}

The branching process in the TNS algorithm is based on the concept of a \emph{Galton-Watson tree}~\cite{Galton1875}, which is a finite rooted tree recursively generated as follows. At each step, a number of offspring is independently sampled from a distribution. Since a geometric tree consists only of bifurcations, terminations, and continuations, the accepted values for the number of offspring are: zero (termination), one (continuation), and two (bifurcation). The Galton-Watson algorithm generates only a combinatorial tree, with no embedding in space, so we modify the traditional process to introduce a dependency of the tree growth on the embedding, so that the bifurcation/termination probabilities depend on the path distance of the growing tip from the root.

The bifurcation/termination step of the growth process of a geometric tree with associated barcode $B$ proceeds as follows. Each growing tip of the tree is assigned a bar $(b_i,d_i)$ sampled from the barcode $B$ and a bifurcation angle $a_i$. The growing tip first checks the probability to bifurcate, then the probability to terminate. If the growing tip does not bifurcate or terminate, then the branch continues to elongate. The probability to bifurcate depends on $b_i$: as the distance from the root to the growing tip approaches $b_i$, the probability to bifurcate increases exponentially until it attains a maximum of $1$ at $b_i$. Similarly, the probability to terminate depends exponentially on $d_i$. 

The probabilities to bifurcate and terminate are sampled from an exponential distribution $e^{- \lambda x}$, whose free parameter $\lambda$ should be wisely chosen. A very steep exponential distribution (high value of $\lambda$) reduces the variance of the population of geometric trees synthesized based on the same barcode. On the other hand, a very low value of $\lambda$ results in trees that are almost random, since the dependence on the input persistence barcode is decreased significantly. If we assume that growth takes place in discrete steps of size $L$, the value of the parameter $\lambda$ should be of the order of the step size $L$, to ensure biologically appropriate variance \cite{tns}. Assuming $L=1$ in some appropriate units, we usually select  $\lambda \approx 1$, so that the bifurcation and termination points are stochastically chosen but still strongly correlated with the input persistence barcodes.

Contrary to other neuron synthesis algorithms~\cite{Koene2009} that sample the branching and termination probabilities from independent distributions, in the TNS the correlation of these probabilities is captured in the structure of the barcode. When the growing tip bifurcates, the corresponding bar is removed from the input barcode to exclude re-sampling of the same conditional probability, thus recording the tree's growth history, which is essential for reproducing the branching structure. In the event of a termination, the growing tip is deactivated, and the bar that corresponds to this termination point is removed from the reference barcode. 

At a bifurcation, the directions of the two daughter branches created depend on the bifurcation angle $a_i$. In this study, we focus primarily on the combinatorial type and the TMD type of the generated geometric tree, so we do not investigate the effect of bifurcation angles on the growth.

\subsubsection*{Elongation} \label{sec:elongation}

We now describe how the synthesized trees are embedded in $\R^3$. A \emph{segment} of a growing tree is the portion of the tree between a pair of consecutive vertices (parent and child). Each synthesized tree is grown segment by segment. The \emph{direction} of a segment, i.e., the vector $\vec d$ from its starting point to its end point, is a weighted sum of three unit vectors: the cumulative \emph{memory} $\vec m$ of the directions of previous segments within a branch, a \emph{target vector} $\vec t$, and a random vector $\vec r$~\cite{Koene2009}. The memory term is a weighted sum of the previous directions of the branch, with the weights decreasing with distance from the tip. As long as the memory function decreases faster than linearly with the distance from the growing tip, the exact choice of function is not important \cite{tns}. The target vector is chosen at the beginning of each branch and depends on the bifurcation angles. The random component is a unit vector sampled uniformly from $\R^3$ at each step.  The direction of the segment 
$$\vec d = \rho \vec r + \tau \vec t + \mu \vec m,$$ 
then depends on three weight parameters $\rho$, $\tau$, and $\mu$, where $\rho + \tau + \mu = 1$.

An increase of the randomness weight $\rho$ results in a highly tortuous branch, approaching the limit of a simple random walk when $\rho=1$. If the targeting weight $\tau = 1$, the branch will be a straight line in the target direction. Different combinations of the three parameters $(\tau, \rho, \mu)$ generate more or less meandering branches and thus reproduce a large diversity of geometric trees.  

\subsubsection*{The Elder Rule and TNS}

The TNS provides a sort of right inverse to the $\tmd$. To recreate a tree that is close to TMD-equivalent to the original, the branch corresponding to a particular bar $(b_i,d_i)$ in the barcode can be attached only to  branches corresponding to bars $(b_j,d_j)$ such that $d_i < d_j$ and $b_i > b_j$. This rule ensures that the Elder rule (at a bifurcation, the longer component survives) holds in the $\tmd$ transformation. As a result, only a subset of trees with $n$ branches can be generated by the TNS from a given strict barcode with $n$ bars.

\begin{table}[H]
\begin{center}
  \scalebox{0.7}{
\begin{tabular}{l|l|l|}
\cline{2-3}
                                              & \textbf{TMD}                                                      & \textbf{TNS}                                                      \\ \hline
\multicolumn{1}{|l|}{\textbf{Goal}}           & Compute the barcode of a tree based on a distance function      & Grow a new tree from a barcode                                    \\ \hline
\multicolumn{1}{|l|}{\textbf{Directionality}} & From leaves to root                                               & From root to leaves                                               \\ \hline
\multicolumn{1}{|l|}{\textbf{Domains}}        & $\{\text{geometric trees} \} \longrightarrow \{\text{barcodes}\}$ & $\{\text{barcodes} \} \longrightarrow \{\text{geometric trees}\}$ \\ \hline

\end{tabular}
} 
\end{center}
\caption{
Summary and terminology of the TMD and TNS algorithms. The TMD computes the barcode of a tree from the tips of branches towards the root, whereas the TNS grows the tree in the opposite direction, from the root to the leaves. }

\label{table_tmd_tns}
\end{table}

\section{Tree-realizations of barcodes}

In this section we provide an in-depth analysis of the set of geometric trees that realize a specific strict barcode $B$, i.e., each of which has TMD equal to $B$.

\subsection{Realizing barcodes as trees}

A geometric tree $T$ is a \emph{tree-realization} of a barcode $B$ if $\tmd(T) = B$, i.e., $T\in \tmd^{-1}(B)$. Examples of tree-realizations are provided in Figure~\ref{notation}B, while Figure~\ref{treerez} shows all the possible combinatorial types of tree-realizations of a strict barcode with $n=4$.  

In Figures \ref{notation} and \ref{treerez}, we encode the combinatorial structure of the tree, i.e., how the branches may be attached to each other, in an adjacency matrix in which the $(i,j)$ coefficient is non-zero if the Elder Rule allows bar $i$ to be connected to bar $j$. For example, in Figure~\ref{notation}A, bars $1-3$ may all be connected to the black bar $0$, thus the coefficients $(0,1), (0,2), (0,3)$ are all non-zero in the corresponding adjacency matrix. Note that in each realization only a subset of these possible attachments is actually made (Figure~\ref{notation}B), since each branch can be attached to only one other branch.

The connectivity diagram (bottom of Figure~\ref{notation}A) provides another representation of the pairs of branches that may be connected, in agreement with the Elder Rule. The arrow on an edge in the diagram indicates the direction of the connection. In this example, there are arrows from $0$ towards $1,2,$ and $3$, from $1$ to $2$ and $3$, and from $2$ to $3$.

\begin{figure}[H]
    \centering
    \includegraphics[scale=1.4]{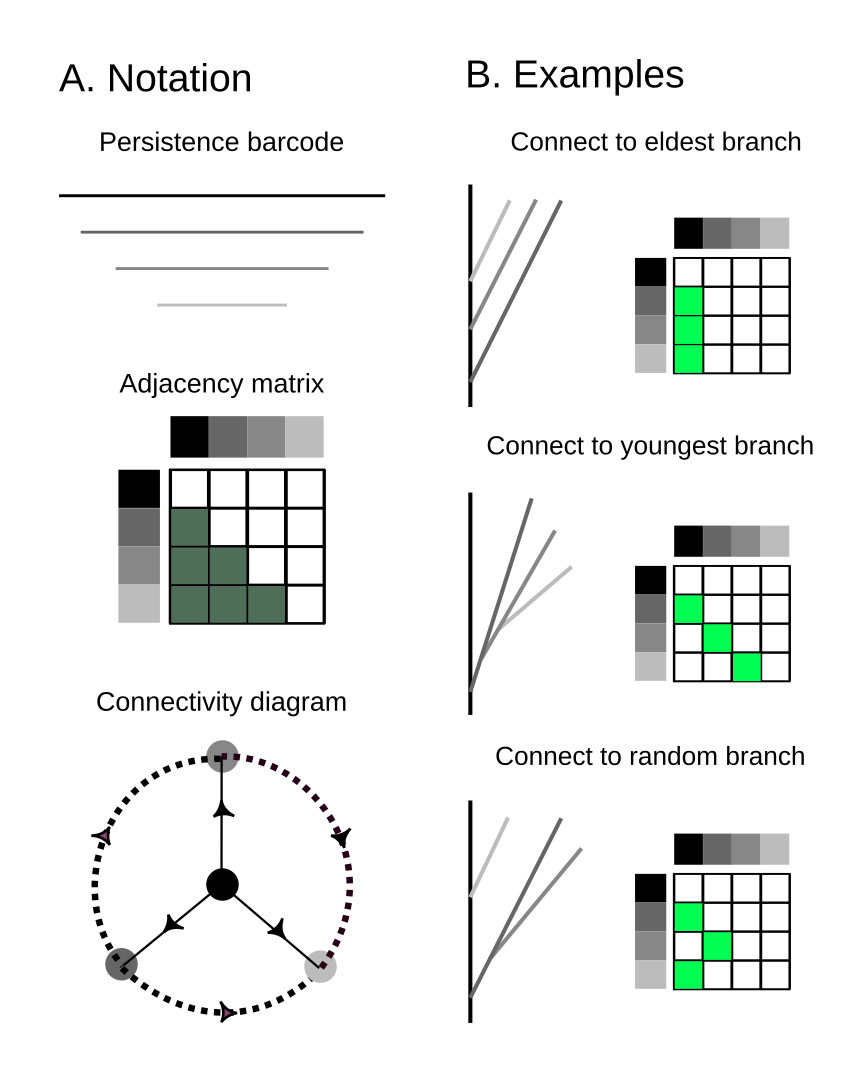}
    \caption{A strict barcode, whose bars are ordered according to birth times (greyscale), defines a unique ordering of death times. This ordering and the Elder Rule constrain the possible combinatorial types of trees that can be realized from this barcode. A. The notation that will be used in this paper from a barcode that corresponds to an adjacency matrix of possible connectivities. Equivalently the possible connectivities are presented in the connectivity diagram. B. Examples of possible tree realizations from brancges that connect to the longest one (top) to random (bottom).}
    \label{notation}
\end{figure}

\begin{figure}[H]
    \centering
    \includegraphics[scale=0.8]{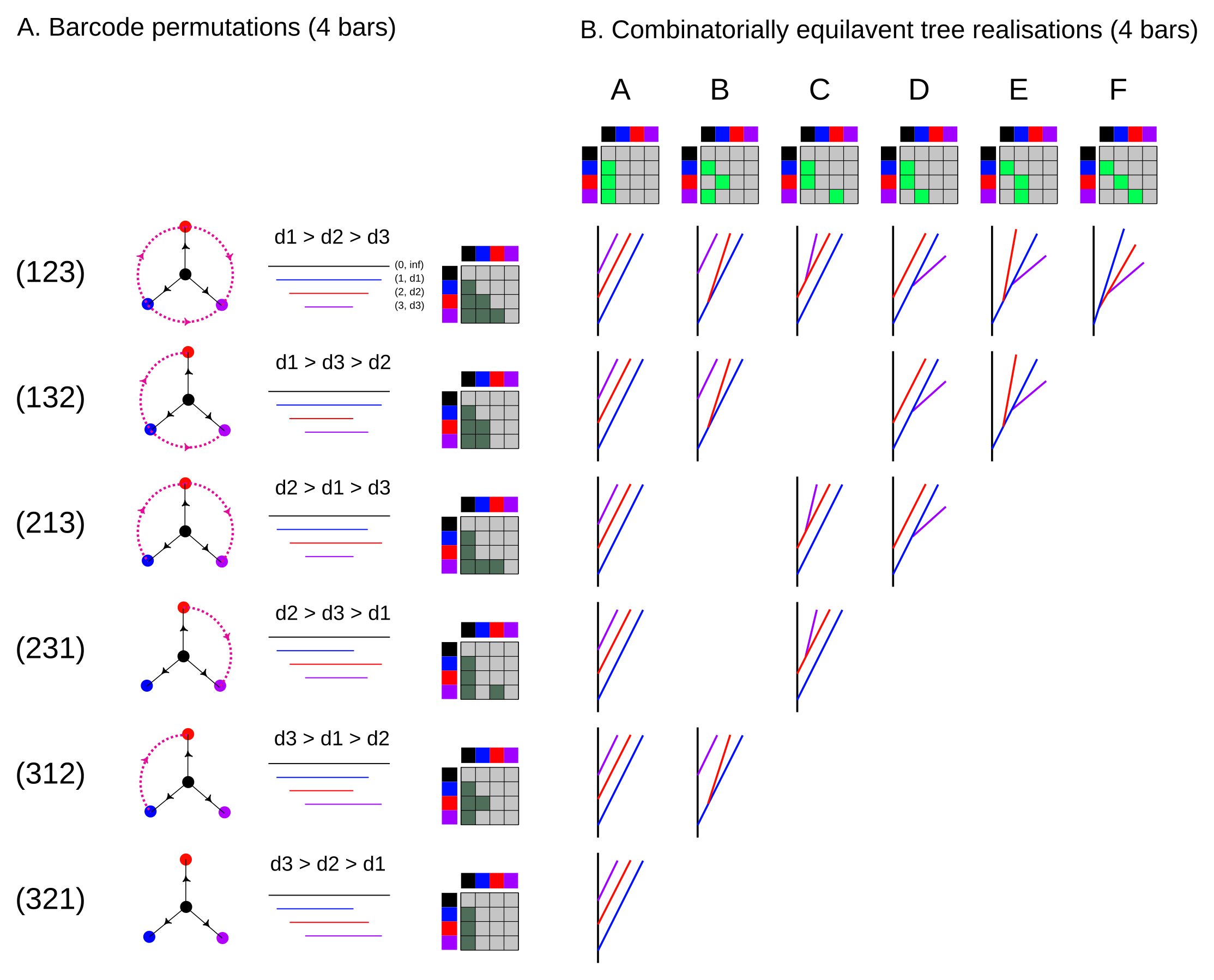}
    \caption{Tree-realizations of  all possible strict barcodes $B$ with four bars. Each row (left) represents a possible permutation of death times for the corresponding order of bars in $B$, i.e., a possible TMD-type. Each barcode can be realized by a subset of all the combinatorial tree types, each represented by a column, with a corresponding adjacency matrix.}
    \label{treerez}
\end{figure}

For any strict barcode $B$, let $\T(B)$ denote the set of combinatorial equivalence classes of tree-realizations of $B$, i.e.,
$$\T(B)= \tmd^{-1}(B)/\csim.$$
We can characterize the equivalence relation on strict barcodes in terms of $\T(B)$.

\begin{lemma}
If $B$ and $B'$ are two strict barcodes with the same number of bars, then
$$ B\bsim B' \;\Longleftrightarrow\;  \T(B) =\T(B').$$
\end{lemma} 

\begin{proof} The order of the deaths in a strict barcode $B$ completely determines the set of combinatorial equivalence classes of its possible tree realizations.

Indeed, the two pairs of bars in Figure~\ref{move_type}(2) lead to the same adjacency possibilities for their respective branches. Only move (1) in Figure~\ref{move_type}, corresponding to switching the order of the deaths of the two bars, 
modifies the permutation equivalence class of the barcode, hence also the set of trees that return the given barcode.

\begin{figure}[H]
\centering
\includegraphics[scale=0.5]{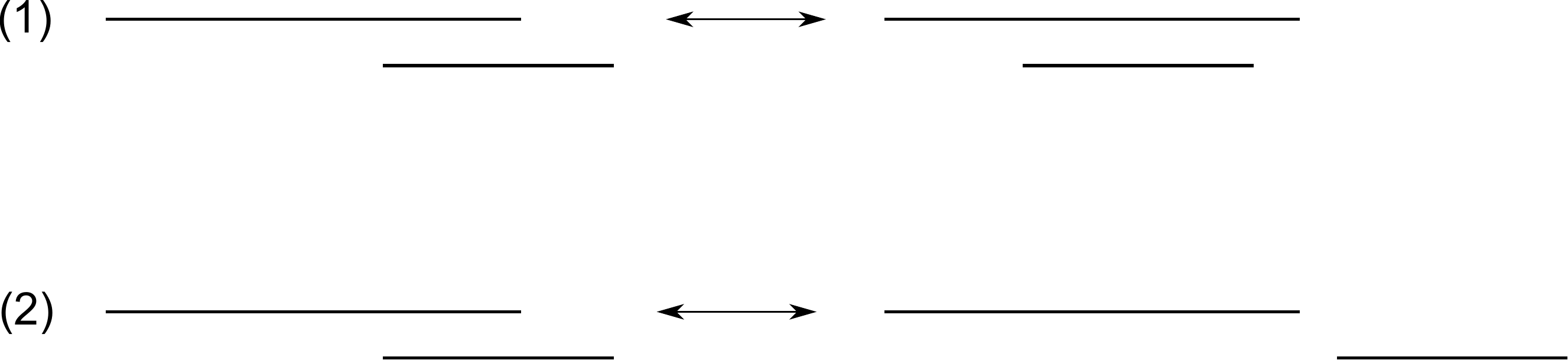} 
\caption{The two possible moves that respect the condition of a realisable barcode. Move $(1)$ modifies the barcode's ordering, whereas move $(2)$ does not change the order of the deaths. }
\label{move_type}
\end{figure}
\end{proof}

\subsection{The combinatorics of tree-realization}\label{sec:combinatorics}

Let $B =\{ (b_i,d_i) \}_{i =0,...,n}$ be a strict barcode. In this section we analyze the \emph{tree-realization number} of $B$, 
$$\trn(B)=\# \T(B),$$
i.e., the number of combinatorial equivalence classes of tree-realizations of $B$.
We provide a formula for $\trn(B)$ in terms of the \emph{index} of each bar $(b_i,d_i)$, i.e.,  the number of bars that include $(b_i,d_i)$ strictly: 
$$\mathrm{index}_i(B)=\# \{j \mid  b_j<b_i<d_i<d_j\}=\#\{j<i\mid d_i < d_j\}.$$
A version of this formula was established by Curry in \cite{curry-2019}.

\begin{lemma}\label{lem:trn} The tree-realization number of a strict barcode $B =\{ (b_i,d_i) \}_{i =0,...,n}$ is equal to the product of the indices of its bars, i.e.,
$$\trn(B)=\prod_{1\leq i\leq n } \mathrm{index}_i(B).$$
\end{lemma}

\begin{proof} Because of the Elder Rule applied in the TMD, one branch can be attached to another only if its corresponding bar is included in the other bar. This simple observation enables us to prove the lemma by a straightforward recursion on the number of bars.\end{proof}

In particular, the maximum tree-realization number for a strict barcode with $n+1$ bars is $n!$, in the specific case where $d_n<...<d_1 <d_0$. We call this case a strictly ordered barcode.

Note that the tree-realization number does \textbf{not} satisfy $$\trn(B) = \trn(B') \implies B \sim B'$$ in general, i.e., the tree-realization number is not a complete invariant of the barcode equivalence relation. For instance, barcodes $(231)$ and $(312)$ in Figure~\ref{cayley_s3} both have $\trn(B) =2$ but have different permutation types. The inverse clearly does hold, however: $$\trn(B) \neq \trn(B') \implies B \not\sim B',$$ 
enabling us to detect non-equivalence of barcodes.
For the pairs of barcodes studied in sections $4$ and $5$ of this paper, though, it is usually true that if they have same TRN, then they are equivalent, so that we can use the TRN to detect equivalence of barcodes in these special cases. In particular, if a new bar is added to a barcode or two deaths are transposed and no other changes take place, then the tree-realization number does change.

Lemma \ref{lem:trn} enables us to quantify how adding a new bar changes tree-realization number.

\begin{lemma}
Let $B = \{ (b_i,d_i) \}_{i =0,...n}$ and $B'=B\cup \{(b_{n+1},d_{n+1})\}$, where $b_{n+1} > b_i$ for all $0\leq i\leq n$, be strict barcodes.  If $d_{i_1} > ... d_{i_{k-1}} > d_{n+1} > d_{i_{k}} > ... d_{i_n}$, then 
$$\trn(B')=\trn(B)\cdot k.$$
\label{lemma_add_bar}
\end{lemma}

\begin{proof}
The condition on $d_{n+1}$ implies that the new bar $(b_{n+1},d_{n+1})$ is included in exactly $k$ other bars, so its index is $k$.
\end{proof}

\begin{example}
Let $B$ be a barcode with four bars such that $b_0 < b_1 < b_2 <b_3$ and $d_0 > d_2 > d_1 > d_3$, i.e., its equivalence class is $(213)$. It is easy to see that $\trn(B)=3$ (see Figure~\ref{example_add_bar}). If we add a new bar $(b_4,d_4)$ such that $d_1>d_4>d_3$, the equivalence class of the new barcode $B'$  is $(2143)$, and bar $(b_4,d_4)$ is included in $(b_0,d_0)$, $(b_2,d_2)$ and $(b_1,d_1)$, but not in $(b_3,d_3)$ because $d_4>d_3$. Therefore, its index is $3$, whence $\trn (B')=3 \cdot 3 = 9$.

\begin{figure}[H]
    \centering
    \includegraphics[scale = 0.5]{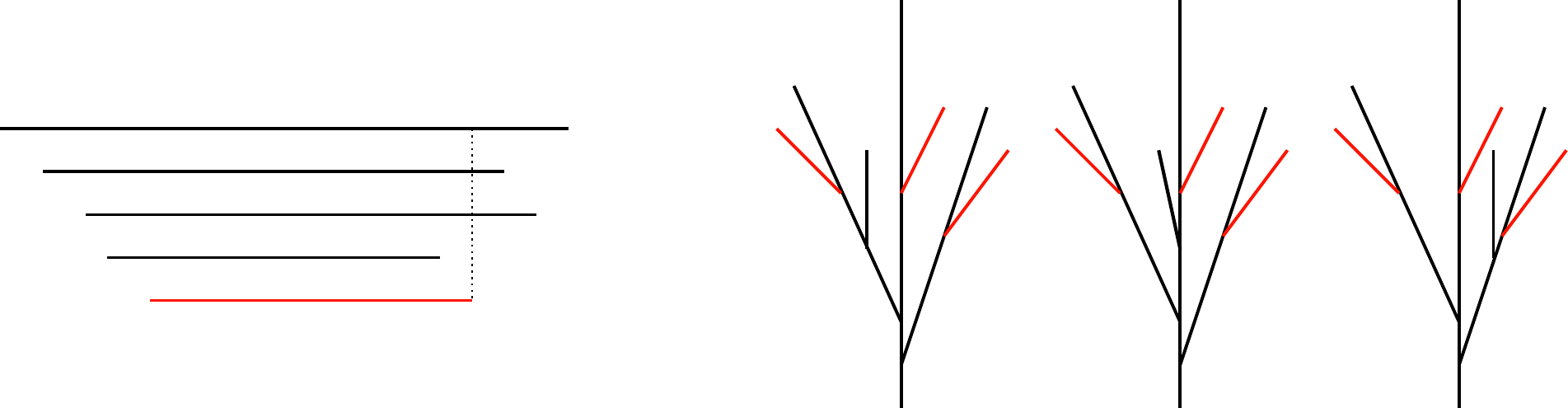}
    \caption{A barcode $B$ in equivalence class $(213)$ is shown in black. There are three possible combinatorial equivalence classes of trees whose TMD barcode is $B$, also represented in black. After adding the extra bar in red, we obtain a new barcode $B'$, in the equivalence class $(2143)$. In a tree-realization of $B'$, the branch corresponding to the  red bar can be attached to any of the branches corresponding to the $0$th, $1$st, and $2$nd bars, represented on the trees by the red branches. This leads to nine possible combinatorial equivalence classes of trees for the barcode $(2134)$.}
    \label{example_add_bar}
\end{figure}
\end{example}

We can also apply Lemma \ref{lem:trn} to determining how switching the order of two consecutive deaths in the barcodes affects the tree realization number. 

\begin{proposition} \label{change_class_real}
Let $B= \{ (b_i,d_i) \}_{i =0,...n}$ be a strict barcode in the equivalence class $(i_1...i_n)$.  Let $B'=\{ (b'_i,d'_i) \}_{i =0,...n}$ be a new barcode obtained by permuting the deaths $d_{i_k}$ and $d_{i_{k+1}}$, i.e., $b_i=b'_i$ for all $i$ and $d_i=d'_i$ for all $i\not= i_k,i_{k+1}$, while $d_{i_k}=d'_{i_{k+1}}$ and $d_{i_{k+1}}=d'_{i_{k}}$. 
\begin{enumerate}
\item If $i_{k} < i_{k+1}$,  then $\mathrm{index}_{i_{k+1}}(B')=\mathrm{index}_{i_{k+1}}(B)-1$, and 
 $$\trn(B')=\frac{\trn(B) (\mathrm{index}_{i_{k+1}}(B)-1) }{\mathrm{index}_{i_{k+1}}(B)}.$$
 
\item If $i_{k} > i_{k+1}$, then $\mathrm{index}_{i_{k+1}}(B')=\mathrm{index}_{i_{k+1}}(B)+1$, and 
 $$\trn(B')=\frac{\trn(B) (\mathrm{index}_{i_{k+1}}(B)+1) }{\mathrm{index}_{i_{k+1}}(B)}.$$
 \end{enumerate}
\end{proposition}

\begin{proof} It is enough to prove (1), since (2) then follows by switching the roles of $B$ and $B'$. 

If $i_{k} < i_{k+1}$, then $b_{i_{k}} < b_{i_{k+1}}$.  Since $B$ is in the equivalence class $(i_1...i_n)$, $d_{i_{k+1}} < d_{i_k} $ as well, whence $(b_{i_{k+1}}, d_{i_{k+1}})\subset(b_{i_k},d_{i_k})$. On the other hand,  $(b'_{i_{k+1}}, d'_{i_{k+1}})\not\subset(b'_{i_k},d'_{i_k})$, but otherwise respects all of the same inclusion relations as $(b_{i_{k+1}}, d_{i_{k+1}})$, so that 
$$\mathrm{index}_{i_{k+1}}(B')=\mathrm{index}_{i_{k+1}}(B)-1,$$
as desired. Moreover, $(b'_{i_{k}}, d'_{i_{k}})\not\subset(b'_{i_{k+1}},d'_{i_{k+1}})$, so $(b'_{i_{k}}, d'_{i_{k}})$ respects exactly the same inclusion relations as $(b_{i_{k}}, d_{i_{k}})$, i.e., 
$$\mathrm{index}_{i_{k}}(B')=\mathrm{index}_{i_{k}}(B).$$
Because no other bars are affected when passing from $B$ to $B'$, we can conclude. 
\end{proof}

\begin{example}
 In Figure~\ref{cayley}, we show all the possible death-transpositions in a strict barcode with five bars. As an example, take $B$ in the equivalence class $(2134)$, so the barcode satisfies $d_2 > d_1 >d_3 >d_4$. The index of $(b_4,d_4)$ is $4$, because it is included in all the other bars. Permuting $d_3$ and $d_4$ leads to a barcode in the equivalence class $(2143)$. The index of the last bar is now $3$ because it is no longer included in the third bar. 
\end{example}

 Recall the bijection $\sigma: \B^{\mathrm{st}}_n \to \mathfrak S_n$ from section \ref{sec:barcodes}. Permuting the order of the deaths $d_i$ and $d_{i+1}$ corresponds to a transposition $(i, i+1)$ in $\mathfrak S_n$ (and to move $(1)$ in Figure~\ref{move_type}). Studying the allowed moves and their effects on the barcode is equivalent to studying the symmetric group seen as generated by transpositions of type $(i, i+1)$, enabling us to create the following revealing visualization of the effect of switching the order of deaths or of adding a new bar by using Cayley graphs of $\mathfrak S_n$.
 
We illustrate the results of Lemma \ref{lemma_add_bar} and Proposition \ref{change_class_real} on the two Cayley graphs of $\mathfrak
 S_3$ and $\mathfrak S_4$ in Figures \ref{cayley_s3} and \ref{cayley}.
 Figure~\ref{cayley_s3} shows the Cayley graph of $\mathfrak S_3$ generated by the permutations $(12)$ and $(23)$ and the corresponding equivalence classes of  barcodes. The vertices of the graph correspond  to the permutations in the symmetric group and their corresponding barcode types, and the edges between them to the transpositions transforming one permutation into another. The number next to each bar is its index. The trees that return a given barcode are sketched next to each vertex of the Cayley graph of $\mathfrak S_3$.
 
 \begin{figure}[H]\label{fig:cayley-s3}
    \centering
    \includegraphics[scale = 0.4]{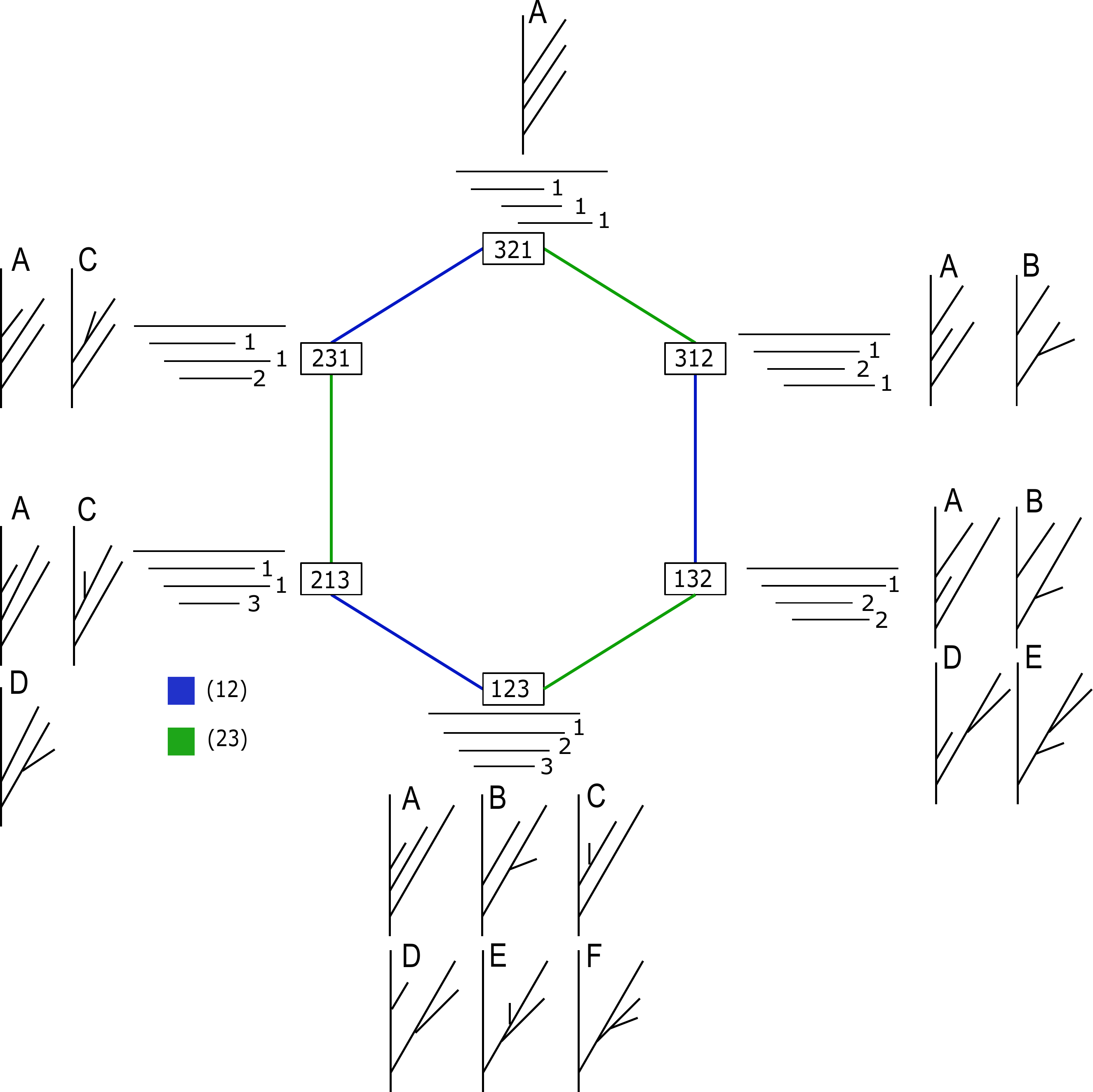}
    \caption{A representation of the space of barcodes with four bars up to permutation equivalence class} as the Cayley graph of $\mathfrak S_3$ generated by the transpositions $(12)$ and $(23)$, respectively. Each vertex is an element of the group. The edges represent the transposition to convert one end point into the other, colored by generator. The number to the right of each bar is its index. All trees $T$ such that $\tmd(T)=B$ are indicated next to a barcode $B$ with the corresponding tree type of Figure~\ref{treerez}. The number of such trees can be computed using Lemma \ref{lem:trn}.
    \label{cayley_s3}
\end{figure}

Figure~\ref{cayley} shows the Cayley graph of $\mathfrak S_4$ generated by the transpositions $(12)$, $(23)$, and $(34)$, illustrating the effect on tree-realization number both of switching the order of two deaths and, in comparison with the previous figure, of adding an extra bar.

\begin{figure}[H]\label{fig:cayley-s4}
\centering
\includegraphics[scale=0.6]{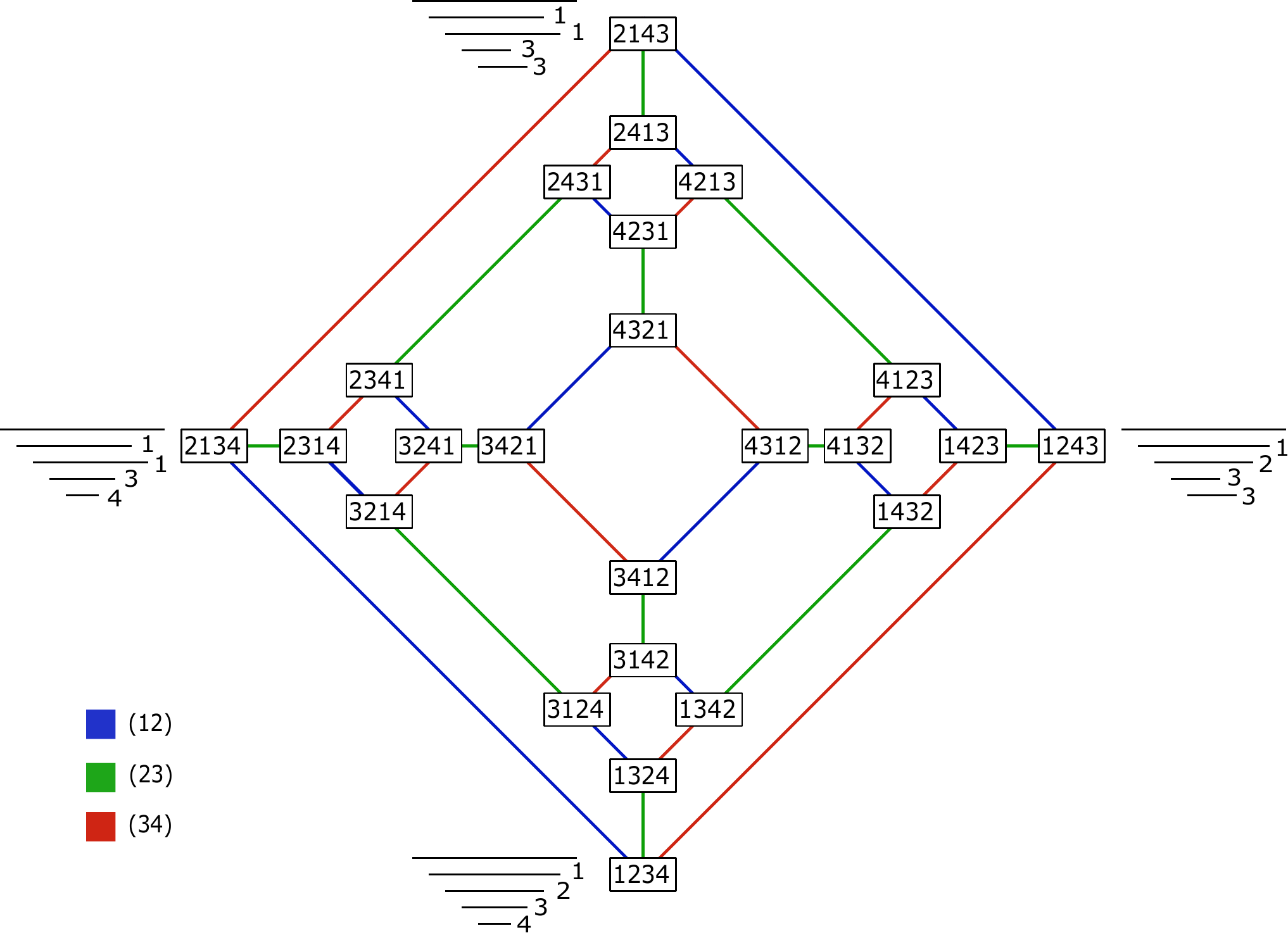} 
\caption{The Cayley graph representing $\mathfrak S_4$ generated by $(12),(23),(34)$. For the four outside elements, we provide the barcode associated to the permutation. The number next to each bar is its index.}
\label{cayley}
\end{figure}

\section{Stability of the TNS}\label{sec:stability}

In this section, we investigate the effect of the composition of the TNS and TMD algorithms from a theoretical perspective. Given a strict barcode $B =  \{(b_i,d_i)\}_{i=0,...,n}$, we apply the TNS  to $B$, for a fixed choice of the parameter $\lambda$ (cf. section \ref{sec:tns}), obtaining a tree $T_B$, and then compute the barcode of $T_B$, $ \tmd(T_B)$, which we denote by $B' = \{(b_i',d_i')\}_{i=0,...,n}$.  To quantify to what extent the TNS acts as an inverse to the TMD, we are interested in determining how similar $B$ and $B'$ are.

Expressing the similarity between $B$ and $B'$ in terms of the bottleneck distance enables us to establish one form of stability for the TNS in the first part of this section.  We establish another type of stability for the TNS in the second part, when we show that the probability that the order of two specific bars will be altered upon applying $\tmd\circ \tns$ decreases exponentially with the distance between the death times of the two bars.  Together these stability results imply that the TNS is an excellent approximation to a (right) inverse to the TMD.

\subsection{Bottleneck stability}

Henceforth, we call the endpoints of the bars of the barcode $B$ \emph{targets}, as the TNS algorithm either creates a new branch or terminates a branch when the distance from the root approaches a birth or death point, respectively. 

By definition of the TNS algorithm (cf. section \ref{sec:tns}), when approaching a target, there is an exponential probability to bifurcate (create a new branch) or terminate, depending on $\lambda$. It follows that for any bar $(b_i,d_i)$ of a given barcode $B$, the distance between $b_i$ and $b_i'$ (the bifurcation distance and the target bifurcation distance of the $i^{\text{th}}$ branch) and the distance between $d_i$ and $d_i'$ (the termination distance and the target termination distance of the $i^{\text{th}}$ branch) should follow an exponential distribution of parameter $\lambda$,
$$ \vert b_i- b'_{i} \vert\sim\textup{Exp}(\lambda) \text{ and }\vert d_i- d'_{i} \vert\sim\textup{Exp}(\lambda).$$

The notion of similarity between barcodes that we consider here is the standard \emph{bottleneck distance}.  Given two strict barcodes $B$ and $B'$ with $n+1$ bars, the bottleneck distance between them is 
\[ d(B,B') = \inf_{\gamma\in \mathfrak S_n}\sup_i \vert b_i- b'_{\gamma(i)} \vert + \vert d_i - d'_{\gamma(i)} \vert. \] 

\begin{lemma}\label{bottleneck_stability_lemma}
Let $B$ be a strict barcode with $n$ bars, and let $B' = \tmd \circ \tns (B)$. If $B \sim B'$, then
\begin{equation}\label{eq_proba_bottleneck}
\mathbb{P}\big(d(B,B') > \varepsilon \big) \leq 1-(1-\exp(- \lambda \varepsilon)(\lambda \varepsilon +1))^n. \end{equation}
\end{lemma}

\begin{proof}Considering the case where $\gamma$ is the identity, we see that  
\[ d(B,B') \leq \sup_i \vert b_i- b'_{i} \vert + \vert d_i - d'_{i} \vert. \]  If $B \sim B’$, the differences between the new and original values of the births and deaths all follow an exponential distribution, 
$$ \vert b_i- b'_{i} \vert\sim\textup{Exp}(\lambda) \text{ and }\vert d_i- d'_{i} \vert\sim\textup{Exp}(\lambda) .$$ 
 
The cumulative probability distribution function of  $\vert b_i- b'_{i} \vert + \vert d_i - d'_{i} \vert$ is thus given by an Erlang$(2,\lambda)$ distribution
\[\mathbb{P}\big(\vert b_i- b'_{i} \vert + \vert d_i - d'_{i} \vert \leq \varepsilon \big) = 1- (1 +\lambda  \varepsilon )\exp (-\lambda \varepsilon).\] 
 
Because we consider the supremum over $i$ of the sum $\vert b_i- b'_{i} \vert + \vert d_i - d'_{i} \vert$, and all of the $\vert b_i- b'_{i} \vert + \vert d_i - d'_{i} \vert$ are \emph{i.i.d}, it follows from the theory of order statistics that 
\begin{equation}\label{eq_lemma_notcomplement}
\mathbb{P}\big(d(B,B') \leq \varepsilon \big) \geq \mathbb{P}(\sup_i \vert b_i- b'_{i} \vert + \vert d_i - d'_{i} \vert \leq \varepsilon) =  (1-\exp(- \lambda \varepsilon)(\lambda \varepsilon +1))^n. 
\end{equation}
Considering the probability of the complement leads to the result in Equation \ref{eq_proba_bottleneck}.
\end{proof}

Lemma \ref{bottleneck_stability_lemma} implies that the TNS is stable with respect to the bottleneck distance, in a manner dependent on the parameter $\lambda$. To illustrate this dependence, we plot the function of Equation \ref{eq_proba_bottleneck} for different values of $\lambda$ in Figure \ref{bottleneck_proba_fig}. The curve obtained for $\lambda = 1$ (blue in Figure \ref{bottleneck_proba_fig}) makes it clear that setting $\lambda=1$ ensures that the TNS gives rise to a diverse family of new trees that are nonetheless topologically not significantly far from the original ones, which is the desired goal from a biological perspective. Making an appropriate choice of the parameter $\lambda$ is thus essential.

\begin{figure}[H]
    \centering
    \includegraphics[scale=0.7]{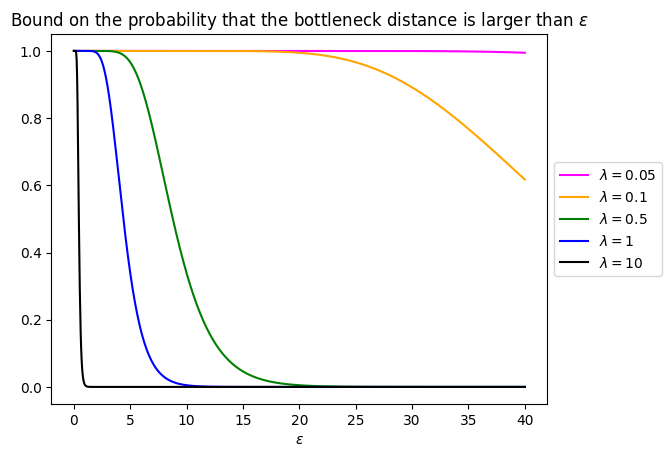}
   \caption{Upper bound on the probability that the bottleneck distance between $B$ and $\tns\circ\tmd(B)$ is larger than $\varepsilon$ (Equation \ref{eq_proba_bottleneck}) for various values of $\lambda$ and for $n=10$. }
    \label{bottleneck_proba_fig}
\end{figure}

If $B \sim B'$, the bound by $\gamma = I$ in the formula for the bottleneck distance is computed between pairs of points that follow the same exponential law, as the order of bars is preserved. If $B \not\sim B'$, for example when a switch of bars occurs, then we cannot assume that the distances between matched pairs of points in the computed bottleneck distance follow the same law. 
Change of permutation type between $B$ and $B'$ is more frequent for small~$\lambda$ (Figure~\ref{stability_swaps}). Therefore, the previous lemma is usually not applicable  for small values of $\lambda$, for which it is any case not particularly useful, as shown in Figure \ref{bottleneck_proba_fig}. In Figure \ref{erlang_balls} we summarize graphically the discussion above. The transposition of bars is studied in detail in the next section.


 \begin{figure}[H]
    \centering
    \includegraphics[scale=0.5]{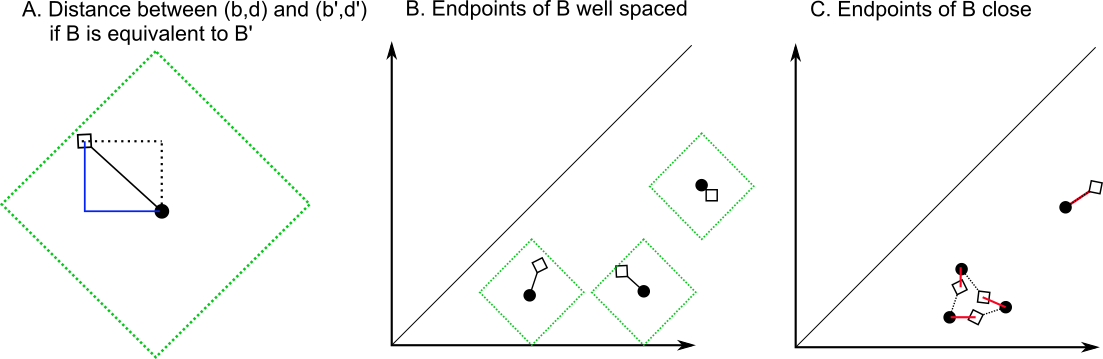}
    \footnotesize
    \put(-380,85){$\sim$ Exp$(\lambda)$}
    \put(-352,65){$\sim$ Exp$(\lambda)$}
    \put(-400,43){\color{blue}$\sim$ Erlang$(2,\lambda)$}
    \caption{A. The $\ell_1$-distance between the black bullet and the diamond follows an Erlang$(2,\lambda)$ distribution. The interior of the green square defines a bound for the $\ell_1$-distance from the black bullet that depends on the value of the parameter~$\lambda$. B. If the endpoints of the bars of $B$ are sufficiently far away from each other and $B\sim \tmd \circ \tns(B)$, then, with high probability, taking $\gamma =I$ will minimize the $\ell_1$-distance between pairs of endpoints of bars. C. If the endpoints of $B$ are instead close to each other, then it is more likely that $B\not\sim \tmd \circ \tns(B)$, so that the optimal choice of $\gamma$ (represented by red segments) is not the identity. The red distances do not necessarily follow exponential distributions, so the proof of Lemma \ref{bottleneck_stability_lemma} does not apply.}
    \label{erlang_balls}
\end{figure}

We perform two experiments to illustrate our theoretical results computationally. First, we compute the bottleneck distance between input barcodes $B$ and output barcodes $B'$, for increasing values of lambda $\lambda$ from  $0.01$ to $2$ (see Figure~\ref{stability_BT}A). The computational results (average bottleneck distances in red, Figure~\ref{stability_BT}A2) fit the curve of the expected mean of the probability density function\footnote{The PDF can be deduced from Equation \ref{eq_lemma_notcomplement} in the proof of Lemma \ref{bottleneck_stability_lemma} by taking the derivative of $(1-\exp(- \lambda \varepsilon)(\lambda \varepsilon +1))^n$.} well (blue curve). 

We also compute the cumulative density function for $0 \le \epsilon \le 200$ and $0 \le \lambda \le 2$, which we compare to the computational results (red points, Figure~\ref{stability_BT}A3), showing that they match the theoretical prediction (blue colormap) very closely for a wide range of sufficiently large $\lambda$ (zoom-in, Figure~\ref{stability_BT}A3). However, for very small values of $\lambda$, the condition $B \sim B'$ is not always satisfied, leading to the observation that for $\lambda < 0.2$, the computationally computed bottleneck distances are larger than the theoretically expected values.

Second, we compute the bottleneck distance between input and output barcodes for various fixed values of $\lambda$, where the input barcodes arise by gradually decreasing the death time of one bar of an initial barcode $B$ and thus increasing the distance to the next death time in the sequence (see Figure~\ref{stability_BT}B). All other bars of the initial barcode $B$ remain the same. We observe that the bottleneck distance depends only on the value of $\lambda$ and not on the distance between the bars of the input barcode.

\begin{figure}[H]
\centering
\includegraphics[scale=0.34]{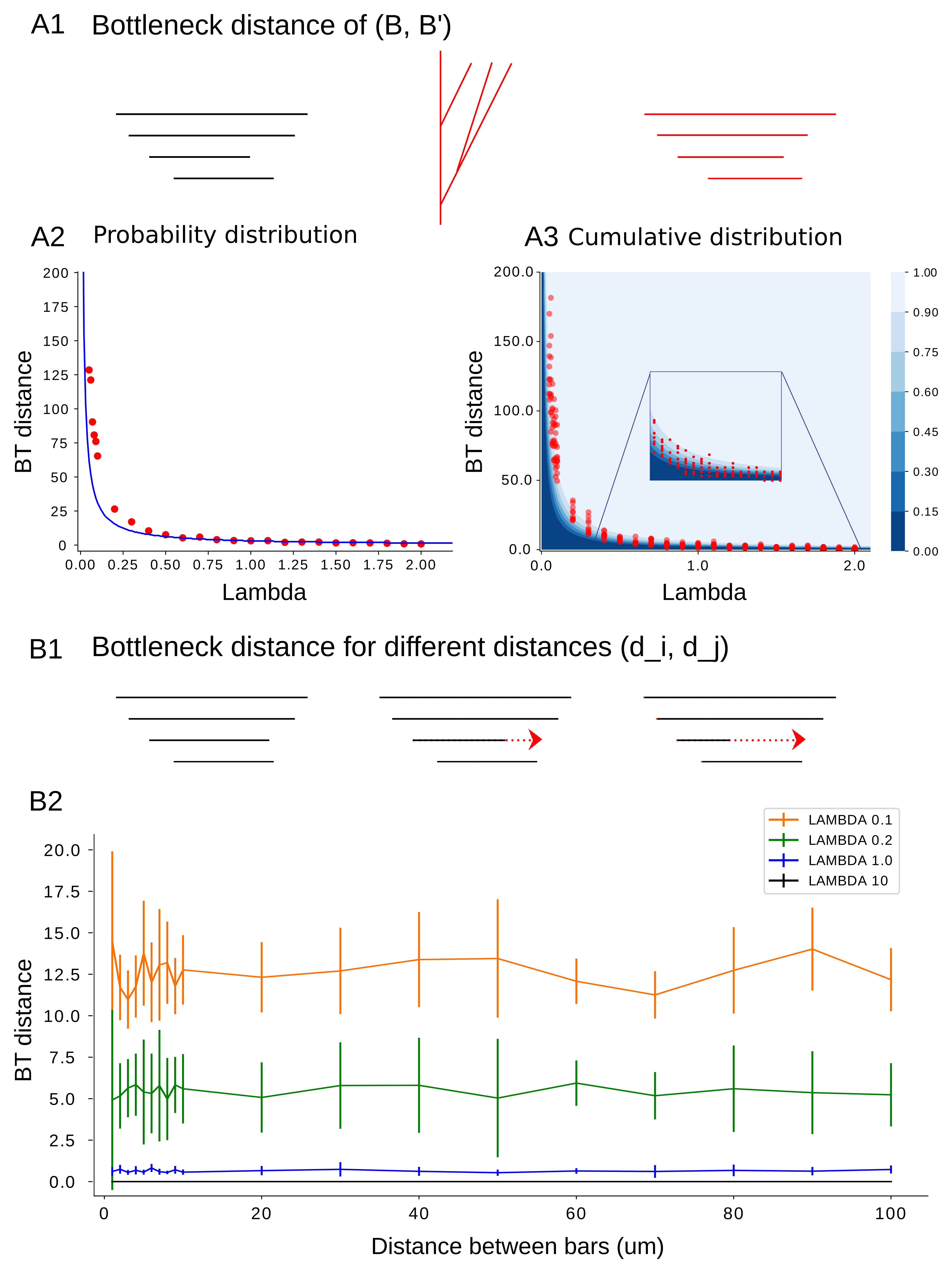} 
\caption{A. Bottleneck distance as a function of $\lambda$. We compute the bottleneck distance between an input barcode $B$ and an output barcode $B'$ for $\lambda = 0.01 - 2$. A1. From barcode $B$ (in black), a tree (in red) is generated using the TNS which results in a new barcode $B' = \tmd \circ \tns (B)$ (in red). A2. The average bottleneck distance (red points) is compared to the expected mean of the probability distribution function found in Lemma~\ref{bottleneck_stability_lemma} (blue curve). A3. The bottleneck distances (red) are compared to the cumulative distribution probability for $0< \epsilon < 200$ and $0< \lambda < 2$ (blue). B. Bottleneck distance between $B$ and $B'$ as a function of distances between bars in $B$. B1. We consider barcodes of the same permutation type for different distances between two bars $(b_{i},d_{i})$ and $(b_{j}, d_{j})$ of the initial barcode $B$ that are consecutive in the order of deaths. B2. For each input barcode with increasing $d_{i}$, distance between death times presented in $x$-axis, $100$ synthesized barcodes are generated and the bottleneck distance between the input and output barcodes is computed ($y$-axis), which depends only on the value of $\lambda$ and not on the distance between the bars.}
\label{stability_BT}
\end{figure}

\subsection{Transposition stability}\label{sec:switchclass}

As the TNS algorithm is a stochastic process, the image of any strict barcode $B= \{(b_i,d_i\}$ under the composite $\tmd \circ \tns$ essentially always differs at least slightly from $B$. Here we determine the probability that the orders of the death times of two specific bars of $B$ and $\tmd\circ \tns(B)=B'=\{(b_i',d_i')\}$ are different, so that $B$ and $\tmd\circ \tns(B)$ are not combinatorially equivalent, i.e., the associated permutations are different, as long as the birth times are not also transposed.

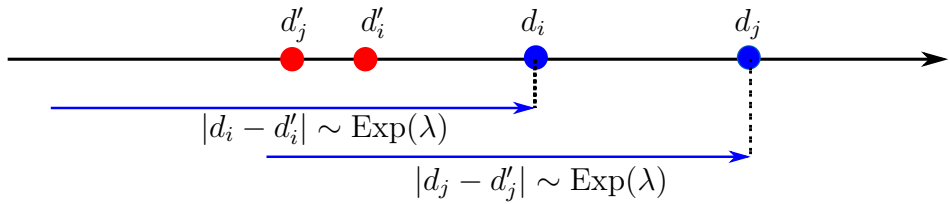
\begin{figure}[H]
    \centering
    \includegraphics[scale = 0.6]{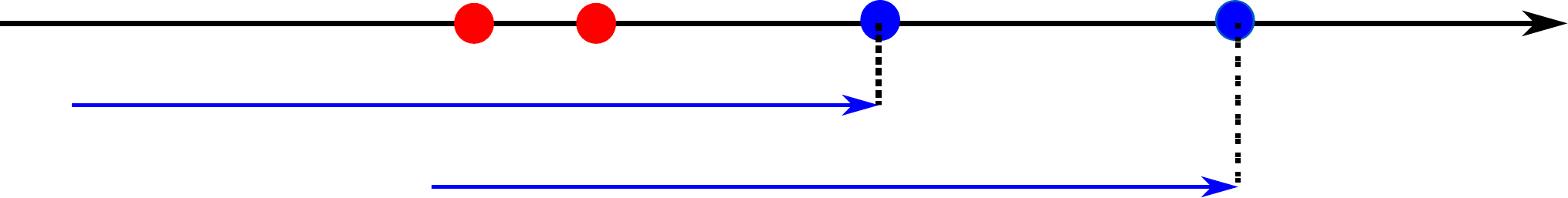}
    \put(-220,50){$d_i'$ }
    \put(-250,50){$d_j'$ }
    \put(-160,50){$d_i$ }
    \put(-80,50){$d_j$ }
    \put(-280,10){$|d_i - d_i'| \sim \text{Exp}(\lambda)$ }
    \put(-200,-10){$|d_j - d_j'| \sim \text{Exp}(\lambda)$ }
    \caption{ We are interested in the case where $d_j'<d_i'$ when we start from $d_i<d_j$. The distances $|d_i - d_i'|$ and $ |d_j-d_j'|$ both follow an exponential law of parameter $\lambda$. The probability to terminate increases exponentially when approaching $d_i$ and $d_j$, as represented by the blue arrows. }
    \label{didj}
\end{figure}

\begin{lemma} \label{switch_class}
Let $B$ be a strict barcode, and let $(b_i, d_i), (b_j, d_j)$ be bars of $B$ such that $d_i < d_j $. Let  $(b'_i,d_i')$ and $(b'_j,d_j')$ denote the corresponding bars in $B'=\tmd\circ \tns (B)$. The probability that $d_j'<d_i'$ is \[\mathbb{P}(d_j' < d_i') = \frac{1}{2} \exp(-\lambda (d_j-d_i)).\]
\end{lemma}

The TNS thus exhibits a sort of ``transposition stability'':  the probability that the death times of two bars will be transposed decreases exponentially with the distance between those death times.

\begin{proof}
We compute $\mathbb{P}(d_j' < d_i') = \mathbb{P}(d_j' < d_i' \,|\, d_i < d_j)$, the probability that $d_j' < d_i'$ given that $d_i < d_j $. Observe first that
\begin{multline*} \mathbb{P}(d_j' < d_i') = \p(d_j + (d_i - d_i') < d_i + (d_j - d_j'))\\ = \p(d_j + X_i < d_i + X_j)  = \p(X_j - X_i > d_j - d_i) .
\end{multline*}

Let $Y = X_j - X_i$. As $X_j$ and $X_i$ both follow an exponential law, the density function of their difference,  $Y$, is given by  $f_Y(t) = \frac{\lambda}{2} \exp(-\lambda t)$ when $t \geq 0$. Therefore, \[\mathbb{P}(d_j' < d_i')  = \p(X_j-X_i > d_j-d_i) = \int_{d_j-d_i}^\infty f_Y(t)dt = \frac{1}{2} \exp(-\lambda (d_j-d_i)).\]
\end{proof}

\begin{remark} \label{stochastic_several_changes}
Since the TNS is based on a stochastic process, multiple transpositions can occur when generating a new tree from a barcode. This makes it challenging to determine the overall probability of changing equivalence classes when computing the composite $\tmd\circ \tns$. Note that the TNS might also affect the birth order, but we will not discuss this possible effect in this paper. For the following experiments, the selected examples do not experience birth-switches, as the cells from which we computed the barcodes were chosen with sufficient gaps between birth values to avoid such switches.
\end{remark}

We perform the following computational experiment to evaluate the transposition stability results. We systematically vary the distance between two bars by changing the death time of a bar in the input barcode and compute the percentage of order changes that occur for different values of lambda(see Figure~\ref{stability_swaps}). We compare the theoretical results (solid lines) to the computational experiment (scatter points) for five different values of lambda. Note that for this experiment, the birth times are chosen to be sufficiently distinct, and only the number of switches due to permutations that correspond to death changes are counted. The experimental results match the theoretical prediction with high accuracy, where we compute the error as the average distance of the computational points from the theoretical curve ($\lambda=10, error=0\%$, $\lambda=5, error=0.02\%$, $\lambda=1, error=0.5\%$, $\lambda=0.5, error=0.9\%$, $\lambda=0.1, error=3\%$, $\lambda=0.05, error=5\%$). Note that the error increases for smaller values of $\lambda$, due to the computational artefacts introduced when $\lambda$ is small.

\begin{figure}[H]
\centering
\includegraphics[scale=0.4]{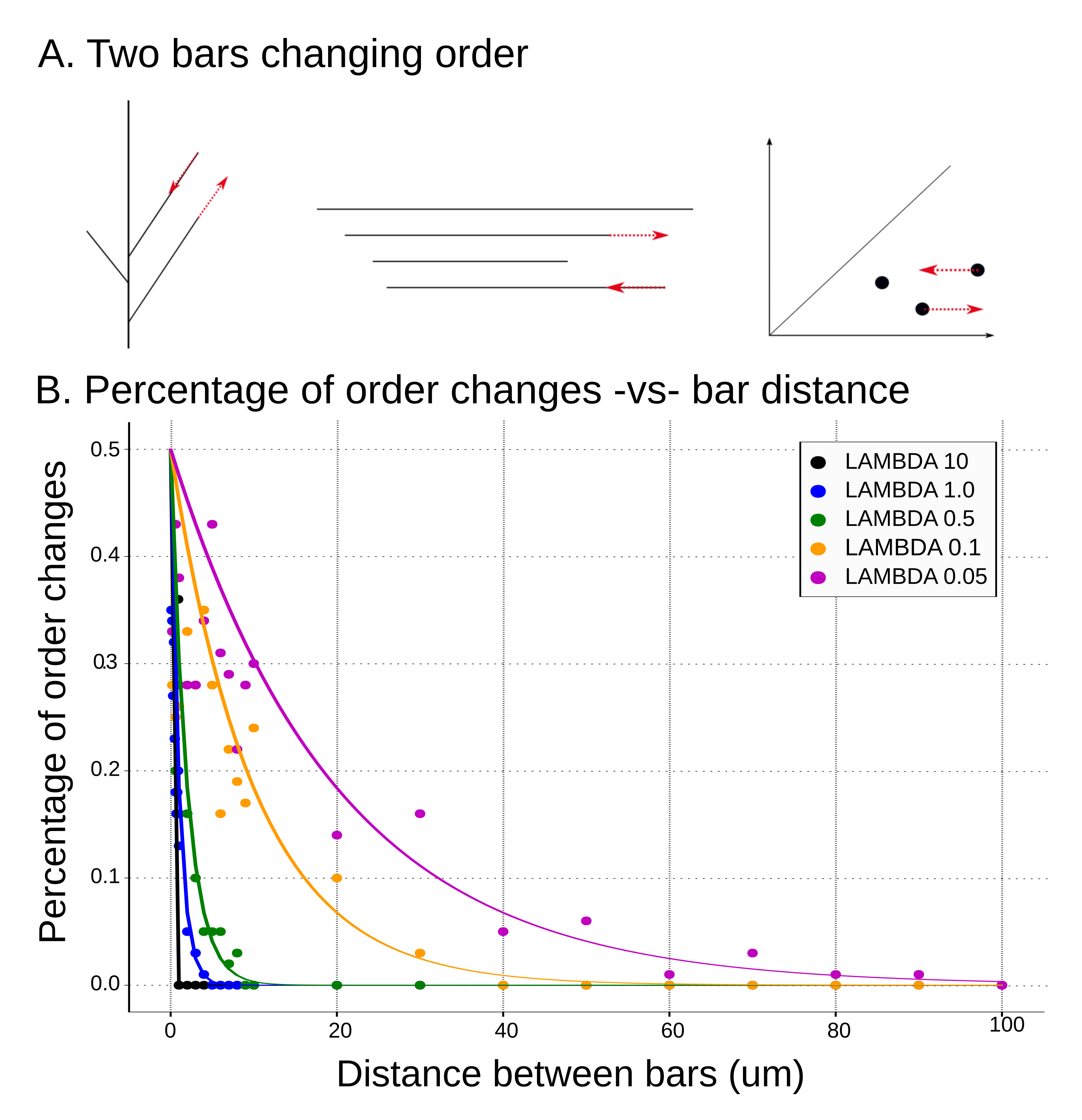} 
\caption{A. Example of two bars changing order, which results in switching of classes. We show here a tree, its barcode, and the corresponding persistence diagram when two consecutive deaths switch their order. The impact of the change is illustrated by the red arrows. B. Percentage of order changes per 100 repetitions for varied distance between death times of two consecutive bars of the input barcode. Comparison of theoretical results (solid lines) to simulations (scatter plot) for different values of lambda.}
\label{stability_swaps}
\end{figure}

\section{Computational exploration of tree-realization}\label{sec:computations}

In this section we present computational results that illustrate the complex relationship between the equivalence class of a barcode and its possible tree-realizations.

We first present four results concerning all geometric trees: a computation of the distribution of tree-realization numbers across the set of equivalence classes of strict barcodes for various numbers of bars, a computation of the empirical distribution of combinatorial types of geometric trees in a synthesized population as a function of the equivalence class of the input barcode, a measurement of the diversity of TMD-equivalence classes among the realizations of a fixed barcode, and  simulations of the fluctations in tree-realization number that can occur as two bars gradually switch the order of their deaths. 

We conclude by reporting on an experiment that sheds light on the distinguishing characteristics of ``biological'' geometric trees, i.e., those that arise from digital reconstructions of neurons.

\subsection{The distribution of tree-realization numbers}

We illustrate here how the number of tree-realizations of strict barcodes with $n+1$ bars depends on $n$. In Figure~\ref{count_reals} we present the distribution of tree-realization numbers across  equivalence classes of barcodes with $n+1$ bars, for $1\leq n \leq 10$. As mentioned in section \ref{sec:combinatorics}, the tree-realization number is maximal for a fixed number of bars if and only if the barcode is strictly ordered. We observe an exponential-like behavior in the distribution of tree-realizations with the increase of the number of bars. We expect to link this behavior with intrinsic properties in the space of barcodes in a future work.

\begin{figure}[H]
    \centering
    \includegraphics[scale=0.4]{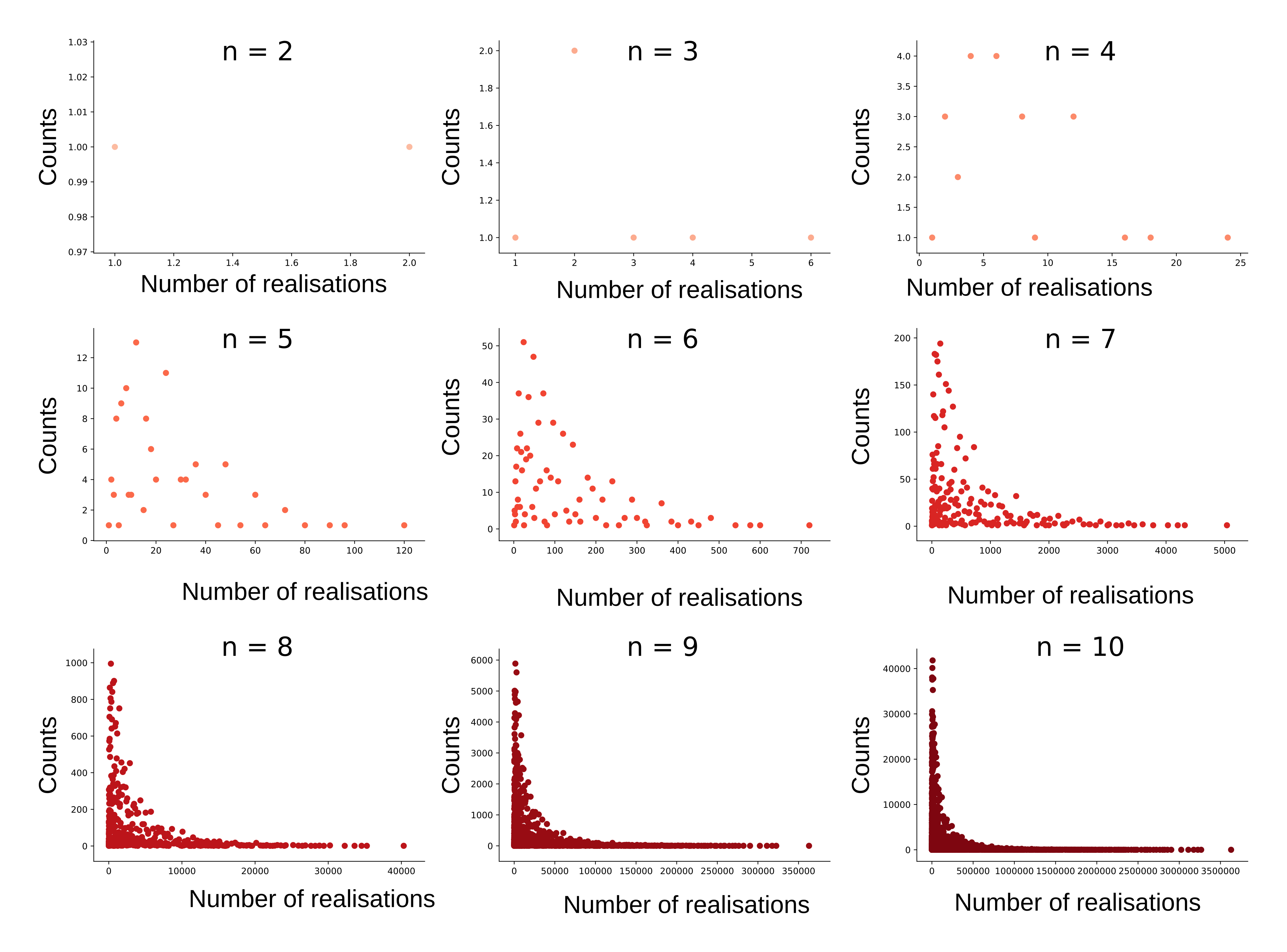}
    \caption{Histogram of tree-realization numbers for equivalence classes of barcodes with $n + 1$ bars ($1\leq n \leq 10$). The maximal tree-realization number for a fixed number of bars can be achieved with exactly one equivalence class, that of the strictly ordered barcode.}
    \label{count_reals}
\end{figure}

\subsection{Empirical distributions of combinatorial types of trees}

In this section, we explore computationally the probability to generate different combinatorial tree types (see Figure~\ref{treerez} A-F) with the TNS. We observe that this probability depends on the choice of the parameter $\lambda$ (cf.~section \ref{sec:tns}). When $\lambda > 2$, the TNS is more likely to generate trees with all branches connected to the longest branch, due to the design of the algorithm. On the other hand, for smaller values of $\lambda$, the probability to generate different types of trees is approximately uniform. 

Focusing on our prefered value of $\lambda$, we generated $1000$ trees for $\lambda = 1$ and computed the percentage of each combinatorial tree type that is realized for each equivalence class of barcodes with four bars (Figure~\ref{real_prob}). There are six possible equivalence classes of strict barcodes with four bars and six combinatorial equivalence classes of geometric trees with four branches.

\begin{figure}[H]
    \centering
    \includegraphics[scale=0.5]{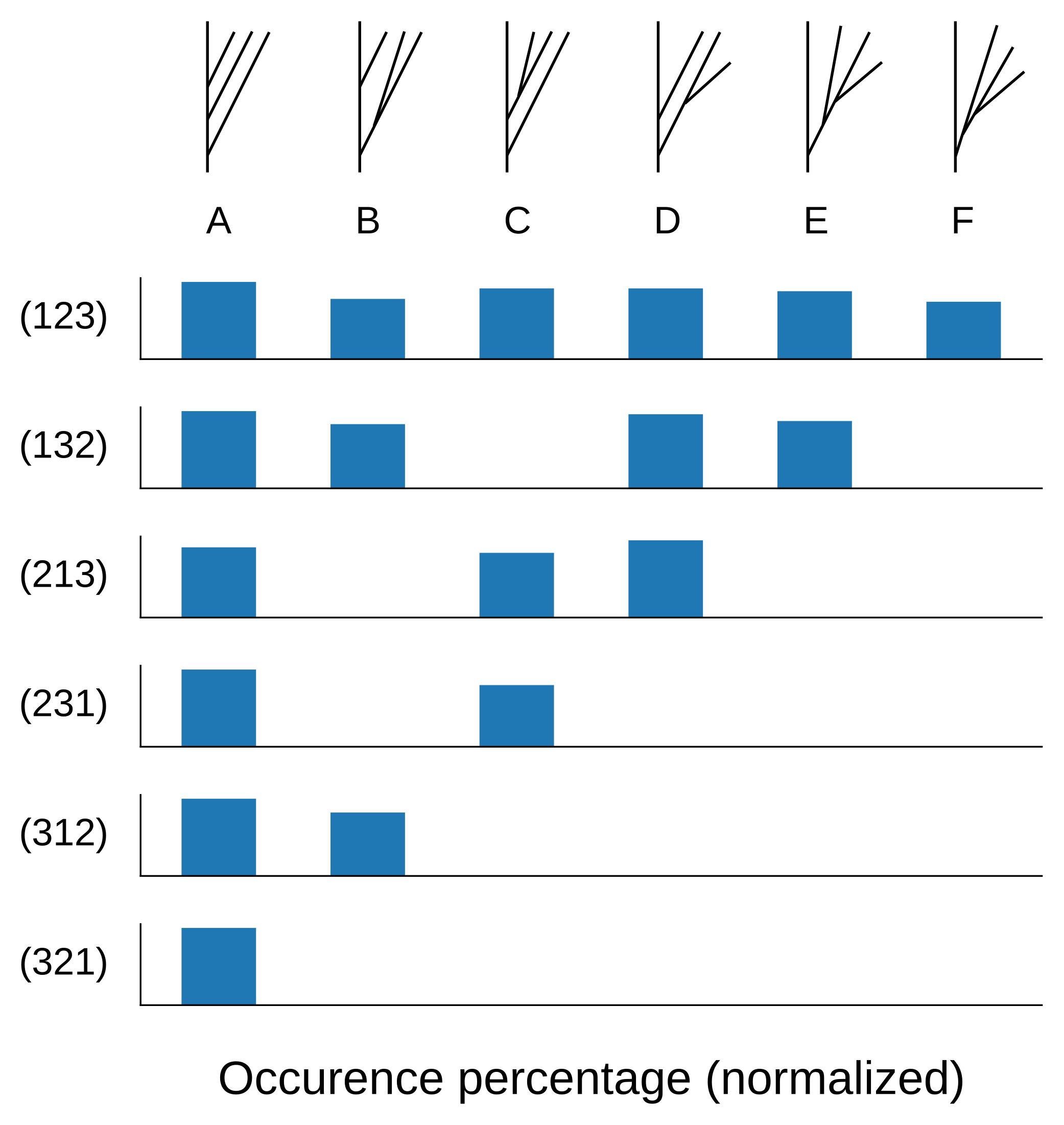}
    \caption{Empirical distribution (percentage of $1000$ trees) of synthesized geometric trees with four branches by combinatorial tree type (columns A - F) for a given input barcode equivalence class (rows), when $\lambda=1$.  We observe that the distribution is approximately uniform.}
    \label{real_prob}
\end{figure}

\subsection{Diversity of realized TMD-equivalence classes}\label{sec:diversity}

We now explore the diversity of TMD-equivalence classes of geometric trees that can be synthesized from a fixed barcode, in the particular case of the TMD of a biologically meaningful tree. For a fixed geometric tree with eight branches arising from a digital reconstruction of layer 4 pyramidal cell, we computed its $\tmd$ barcode, to which we applied the TNS with $\lambda=1$ to generate a set of $100$ geometric trees. We  computed the barcode-type and the persistence diagrams of the synthesized trees (Figure~\ref{real_synth}). 

In agreement with the results presented in Figure~\ref{stability_BT}, the persistence diagrams of the synthesized trees (Figure~\ref{real_synth}B, in blue) are essentially indistinguishable from the persistence diagram of the original barcode (Figure~\ref{real_synth}B, in red). On the other hand, the TMD-equivalence class of a synthesized tree is not necessarily equal to that of the original tree (Figure~\ref{real_synth}A).  Here we represent the TMD-equivalence class of a tree in terms of the permutation $\sigma_B$ corresponding to the equivalence class of its TMD barcode $B$.  In this graphical representation, we plot birth (or bifurcation) index $k$ (on the $x$-axis) versus death (or termination) index $\sigma_B(k)$ (on the $y$-axis).  A strictly ordered barcode would thus correspond to the set of points along the diagonal in this representation.  

\begin{figure}[H]
    \centering
    \includegraphics[scale=0.5]{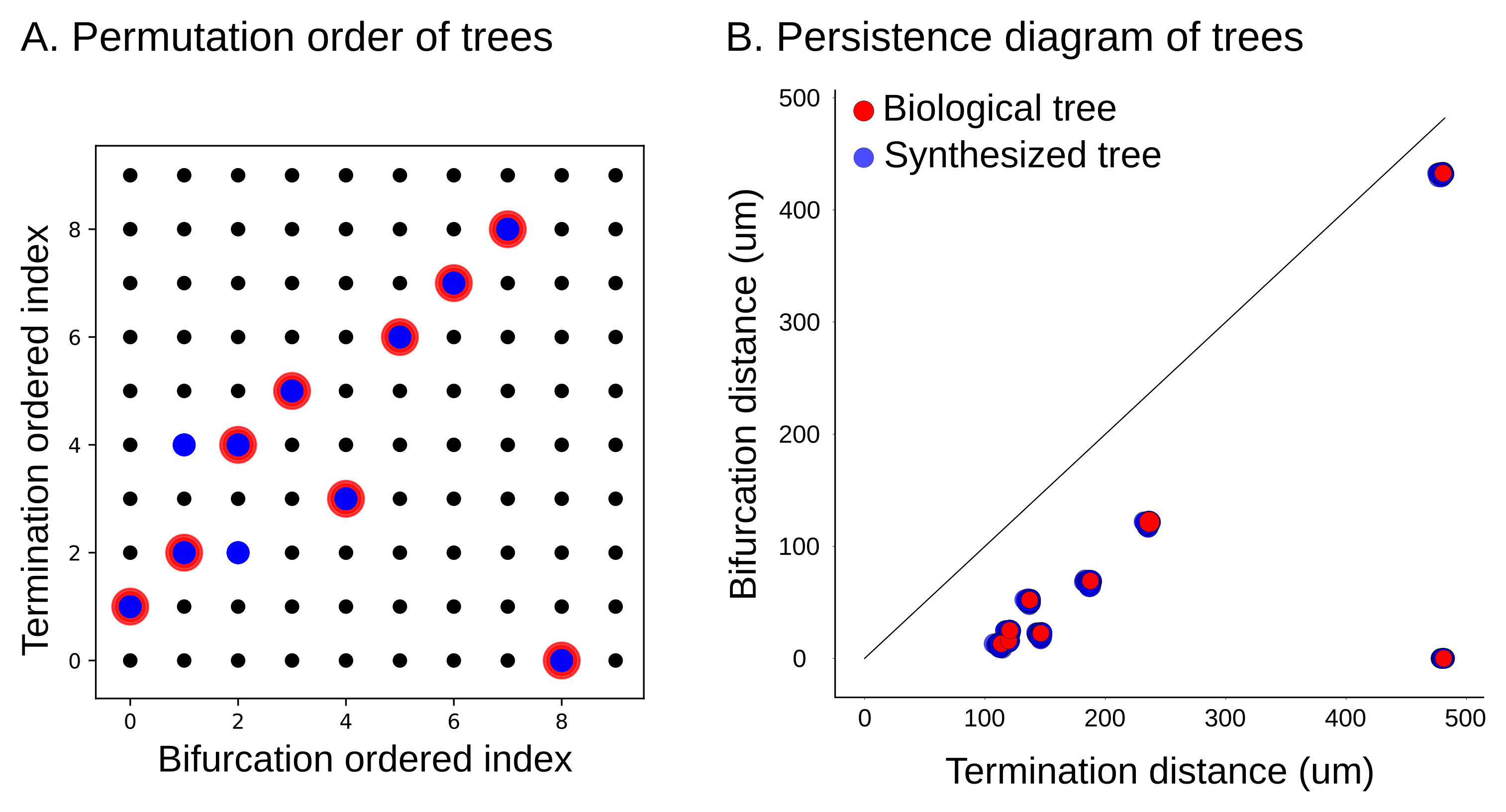}
    \caption{Barcode-equivalence class, represented by the corresponding permutation, (A) and persistence diagram (B) of $100$ synthesized cells based on a geometric tree with eight branches, extracted from a layer 4 pyramidal cell. The barcode-equivalence classes of the synthesized trees (represented by blue dots) can differ from that of the original tree  due to the stochastic nature of synthesis algorithm. The persistence diagrams of the synthesized trees  (B, blue) are essentially indistinguishable from those of the original tree (B, red).}
    \label{real_synth}
\end{figure}

\subsection{Statistics of changing classes}

Motivated by the theoretical results on the probability to change classes in section \ref{sec:switchclass}, we analyze here several simulations of gradual switching of death order of two bars and the resulting effect on tree realizations and their associated barcodes.

Let $B$ be a strict barcode, and let $(b_i, d_i), (b_j, d_j)$ be bars of $B$ such that $d_i < d_j $. By Lemma \ref{switch_class}, for a fixed choice of the parameter $\lambda$ (cf.~section \ref{sec:tns}), the probability that the order of these two bars is reversed in $\tmd\circ \tns (B)$ depends exponentially on the distance between $d_{i}$ and $d_{j}$:  
\[\mathbb{P}(d_j' < d_i') = \frac{1}{2} \exp(-\lambda (d_j-d_i)).\] 
Thus, when the distance between $d_{i}$ and $d_{j}$ decreases, the probability that the order of bars changes increases. When there is no $k$ such that $d_i<d_k<d_j$, Proposition \ref{change_class_real} provides a formula for the tree-realization number of the new barcode obtained when such a switch happens, as long as the order of the birth times is not also reversed.  

We start with a geometric tree $T$ extracted from a digital reconstruction of a neuron and compute its associated barcode $B = \tmd(T)$. We choose two bars $(b_{i},d_{i})$ and $(b_{j}, d_{j})$ of $B$ that are consecutive in the order of deaths and divide the interval $(d_i, d_j)$ into 50 equally sized subintervals. For $0\leq k\leq 50$, let $B_k$ be a barcode that is identical to $B$, except that its $i^{\mathrm{th}}$ bar is $\big(b_i, d_{i} + k(d_{j} -  d_{i})/50\big)$ and its $j^{\mathrm{th}}$ bar is $\big(b_j, d_{j} - k(d_{j} -  d_{i})/50\big) $. An interesting way to visualize this change is to think of $B_k$ as migrating along the edge between $B$ and the barcode with $d_i$ and $d_j$ permuted in the corresponding Cayley graph as $k$ increases. The middle point of the edge corresponds to the non-strict barcode for which the two deaths are equal. 

Let $B'_k=\tmd \circ \tns (B_k)$ for all $k$. Because of the stochastic nature of the TNS algorithm, the  permutation equivalence class of $B'_k$ may be different from that of $B_k.$ Figure~\ref{ex_noise_switch} provides an example of this construction, where the barcodes are represented as persistence diagrams for visualization purposes (cf.~section \ref{sec:barcodes}).



\begin{figure}[H]
    \centering
    \includegraphics[scale = 0.5]{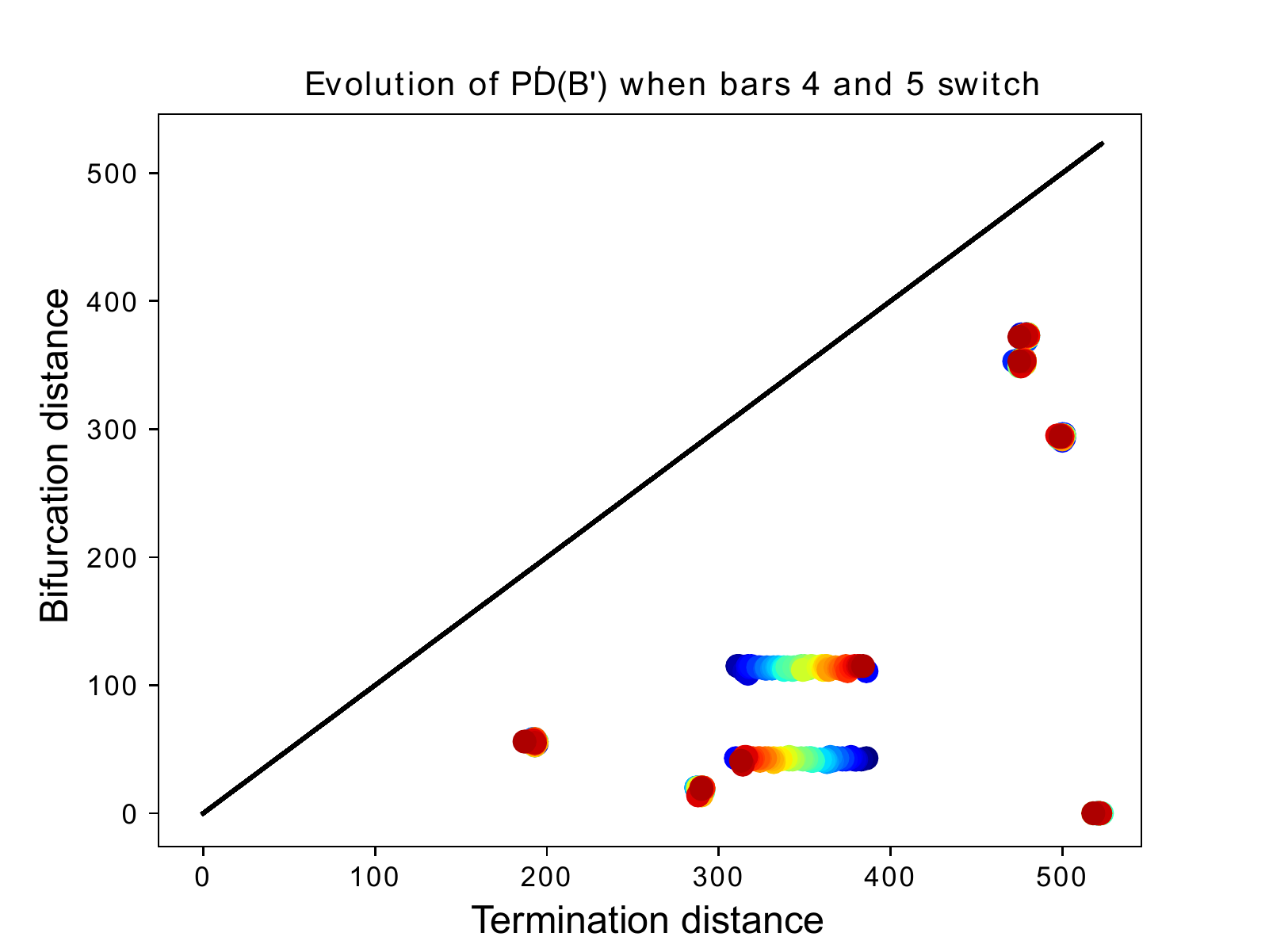}
    \caption{We begin with a barcode with $8$ bars. The death times $d_{i_4}$ and $d_{i_5}$ (i.e., the $4$th and $5$th largest death times) are slowly switching as $k$ increases, represented by red-shifting of the color of the points in the persistence diagram. When $k=0$ (in red), we have the original barcode $B$, and when $k=50$ (in blue) we obtain a barcode identical to the original, except that $(b_{i_4}, d_{i_4})$ is replaced by $(b_{i_4}, d_{i_5})$ and $(b_{i_5}, d_{i_5})$ by $(b_{i_5}, d_{i_4})$.}
    \label{ex_noise_switch}
\end{figure}

To test whether the barcode $B'_k$ is equivalent to the original barcode $B$, we compute its tree-realization number: if $\trn(B'_k)\not= \trn(B)$, then $B$ and $B'_k$ are not equivalent. 
Note that for the specific process that gives rise to $B_k$, it is likely that only the studied death-switch could lead to a difference between the tree-realization numbers of the input and output barcodes, unless two other deaths are too close to each other in the input barcode, as in the last row of Figure~\ref{death_switch}, which we explain further below. Therefore, the tree-realization number provides a very good indication of whether the switch of deaths took place, i.e.,  if $B \sim B_k'$ or not. Indeed, two barcodes that are the same except for two deaths that switched have different tree-realization number, cf. Proposition~\ref{change_class_real}.
Figure~\ref{death_switch} shows several examples of the endpoint-switching process described above and the corresponding evolution of the tree-realization number as $k$ increases. The corresponding permutation type and tree-realization number of each initial barcode, and the bars that are switched  are listed in Table \ref{table_fig_switch_death}.

\begin{table}[H]
\begin{center}
\begin{tabular}{|l|l|l|l|}
\hline
\textbf{}              & \textbf{Permutation}       & \textbf{TRN} & \textbf{Bars that switch} \\ \hline
\textbf{$B^1$}         & $(2 ,6, 8, 1, 5, 7, 4, 3)$ & 810          & $4$ and $5$               \\ \hline
\textbf{$\hat B^1$} & $(2 ,6, 8, 5, 1, 7, 4, 3)$ & 540          & $4$ and $5$               \\ \hline
\textbf{$B^2$}         & $(5, 7, 6, 4, 2, 1, 3)$    & 12           & $2$ and $3$               \\ \hline
\textbf{$\hat B^2$} & $(5, 6, 7, 4, 2, 1, 3)$    & 18           & $2$ and $3$               \\ \hline
\textbf{$B^3$}         & $(5, 7, 6, 4, 2, 1, 3)$    & 12           & $3$ and $4$               \\ \hline
\textbf{$\hat B^3$} & $(5, 7, 4, 6, 2, 1, 3)$    & 18           & $3$ and $4$               \\ \hline
\textbf{$B^4$}         & $(8, 6, 7, 4, 3, 1, 2, 5)$ & 20           & $1$ and $2$               \\ \hline
\textbf{$\hat B^4$} & $(6, 8, 7, 4, 3, 1, 2, 5)$ & 40           & $1$ and $2$               \\ \hline
\end{tabular}
\end{center}
\caption{For each example displayed in Figure~\ref{death_switch}, we list the permutation type and the tree-realization number of the original barcode $B$ and of $\hat B = B_{50}$, and the indices of the bars that are switched. The superscript $i$ in $B^{i}$ indicates the corresponding row of Figure~\ref{death_switch}. For example, the largest death time of barcode $B^1$ is the second bar (in order of birth times), and its shortest death is the third one. When we switch the $4$th and $5$th (from largest to smallest) death times in $B^1$ and $\hat B^1$, the TRN changes from $810$ to $540$.}
\label{table_fig_switch_death}
\end{table}

The top row of Figure~\ref{death_switch} illustrates very well the exponential behavior of changing classes. When the distance between the death times of  the two bars is very small (they are the closest when $k=25$), the tree-realization number oscillates between its values for two different classes and otherwise stays constant. 

The two middle rows come from the same biological tree and hence have the same starting barcode. The difference is that in the second row, the death times of  the two bars are already very close, leading to more frequent changes of equivalence class than in the third row. 

The bottom row illustrates Remark \ref{stochastic_several_changes} well. Since several bars are close to each other (represented here by several points in the persistence diagram that are close to each other), applying the TNS algorithm leads to frequent changes in equivalence classes, leading to the oscillatory behavior of the tree-realization number curve.

\begin{figure}[H]
    \centering
   \includegraphics[scale = 0.35]{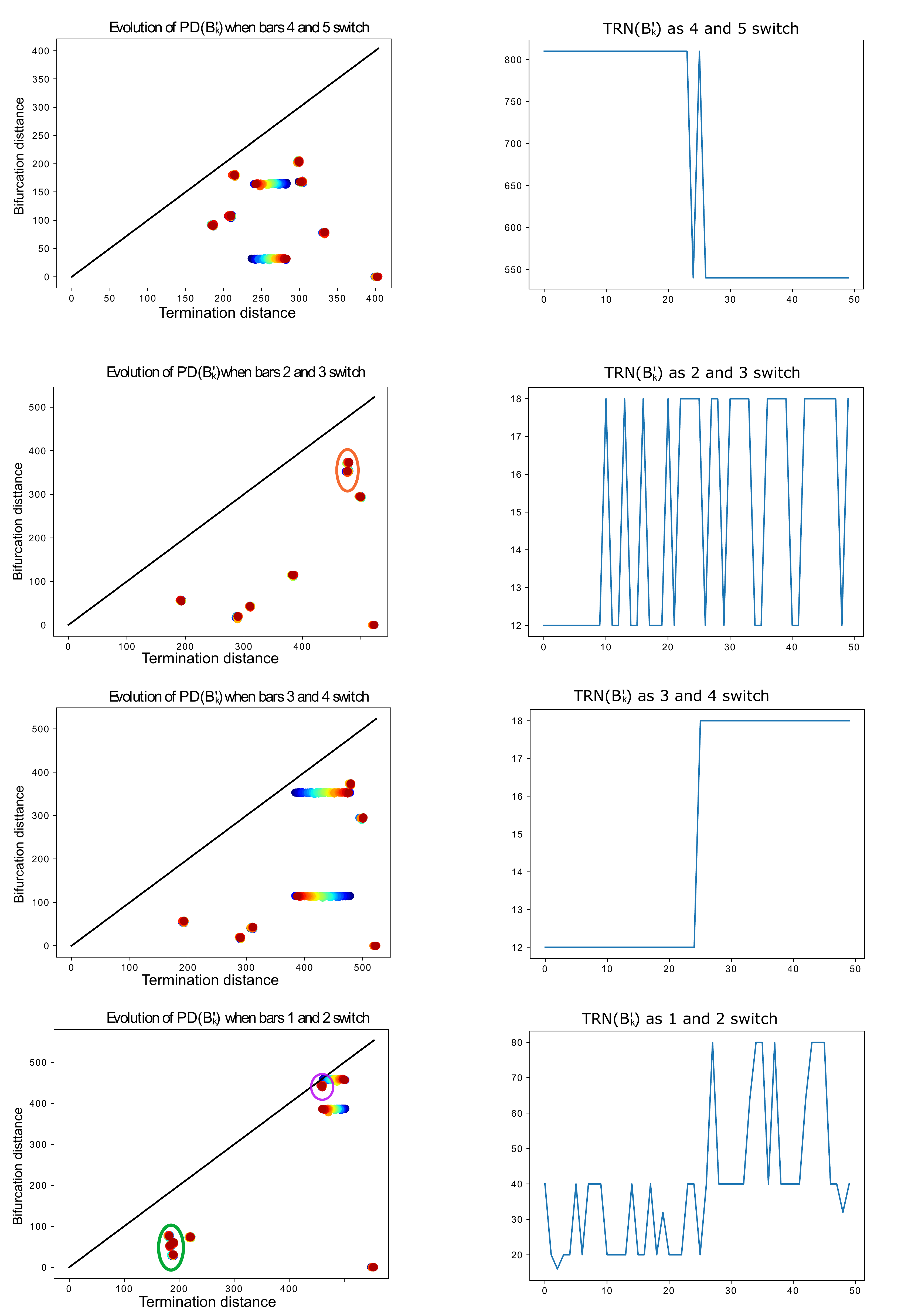}
    \caption{On the left, evolution of $\mathrm{PD}(B_k')$ as $k$ increases (represented by red-shifting of the point color, from red $k=0$ to blue $k=50$), for various pairs of bars. When not clear, we circle in orange the two points that switch. On the right, the corresponding evolution of the tree realization number $\trn(B'_k)$ as $k$ increases.
    For instance, as indicated in Table \ref{table_fig_switch_death}, the tree-realization number of $B^1$ is $810$ and that of $\hat B^{1} = B_{50}^{1}$ is $540$. The barcodes $B_k'$ exhibit the behavior described in Lemma \ref{switch_class}, except for the last row, in which death times that are too close to each other (circled in purple and green) interfere with the process. Without this interference, the tree-realization numbers should oscillate between $20$ and $40$. When $k$ gets close to $50$ (blue), the death time $d_{i_1}$ (largest death time) starts interfering with the third one $d_{i_3}$ (circled in purple) in the tree synthesis process.}
    \label{death_switch}
\end{figure}

\subsection{Tree-realizations of biological barcodes}


Since the original objective in developing the TMD was to classify digital reconstructions of neurons, it is natural to ask whether those barcodes that arise biologically exhibit any special characteristics compared to those arising from other sets of geometric trees. In Figure~\ref{count_real} we employ the graphical representation of permutations introduced in section \ref{sec:diversity} to display as red dots all possible permutations corresponding to TMD-barcodes of biological trees with at most 30 branches arising from a population of digital reconstructions of neurons. Clearly, only a small fraction of the set of all possible permutations can be realized as the barcode-equivalence classes of geometric trees extracted from digital reconstructions of neurons, as every black dot in this plot can arise as a pair $\big(k, \sigma(k)\big)$ for some permutation $\sigma$. In future work, we intend to study the biological relevance of this restriction.

\begin{figure}
    \centering.
    \includegraphics[scale=0.7]{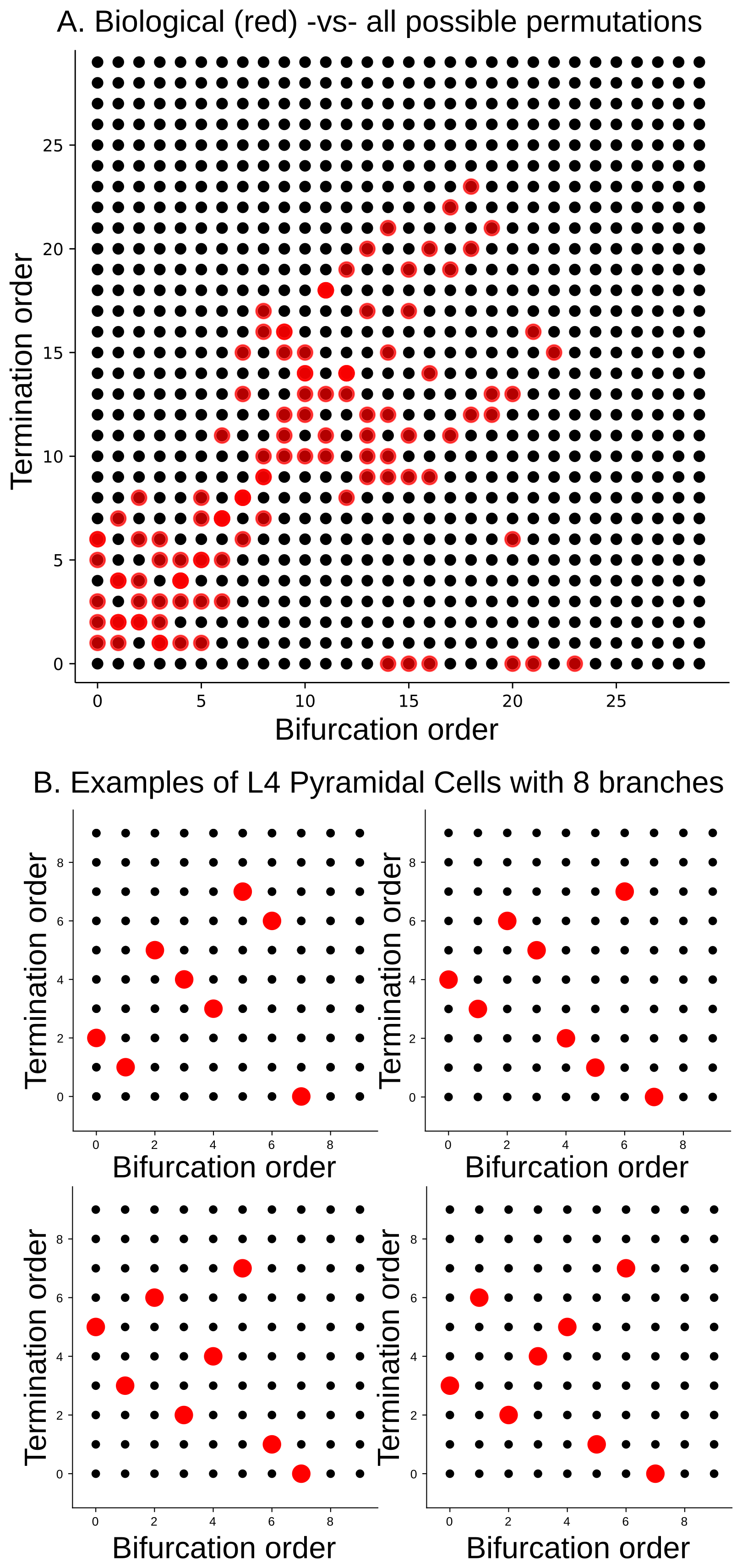}
    \caption{(A) TMD-equivalence classes of a population of biological geometric trees with at most 30 bars (red dots), represented by their associated permutations.  (B) Examples of TMD-equivalence classes of individual biological geometric trees with eight branches, extracted from layer 4 pyramidal cells (red dots).} 
    \label{count_real}
\end{figure}

To provide further insight into the subset of TMD-equivalence classes of biological geometric trees within the set of all possible TMD-equivalence classes, we computed the tree-realization number as a function of the number of bars, for a population of barcodes obtained by applying the TMD to geometric trees extracted from a population of digitally reconstructed neurons. We compared the values obtained to the maximum tree-realization number and to the tree-realization numbers of randomly chosen barcodes with the same number of bars (Figure~\ref{comp_rand_bio}). Interestingly, the barcodes that correspond to apical dendrites (relatively complex neural trees that perform significant processing tasks) exhibit a more narrow range of possible tree-realization numbers than random barcodes of the same size. On the other hand, barcodes of basal dendrites (less complex neuronal trees) exhibit tree-realization numbers similar to those of the randomly generated barcodes.


\begin{center}
\begin{figure}[H]
    \centering
    \includegraphics[scale=0.5]{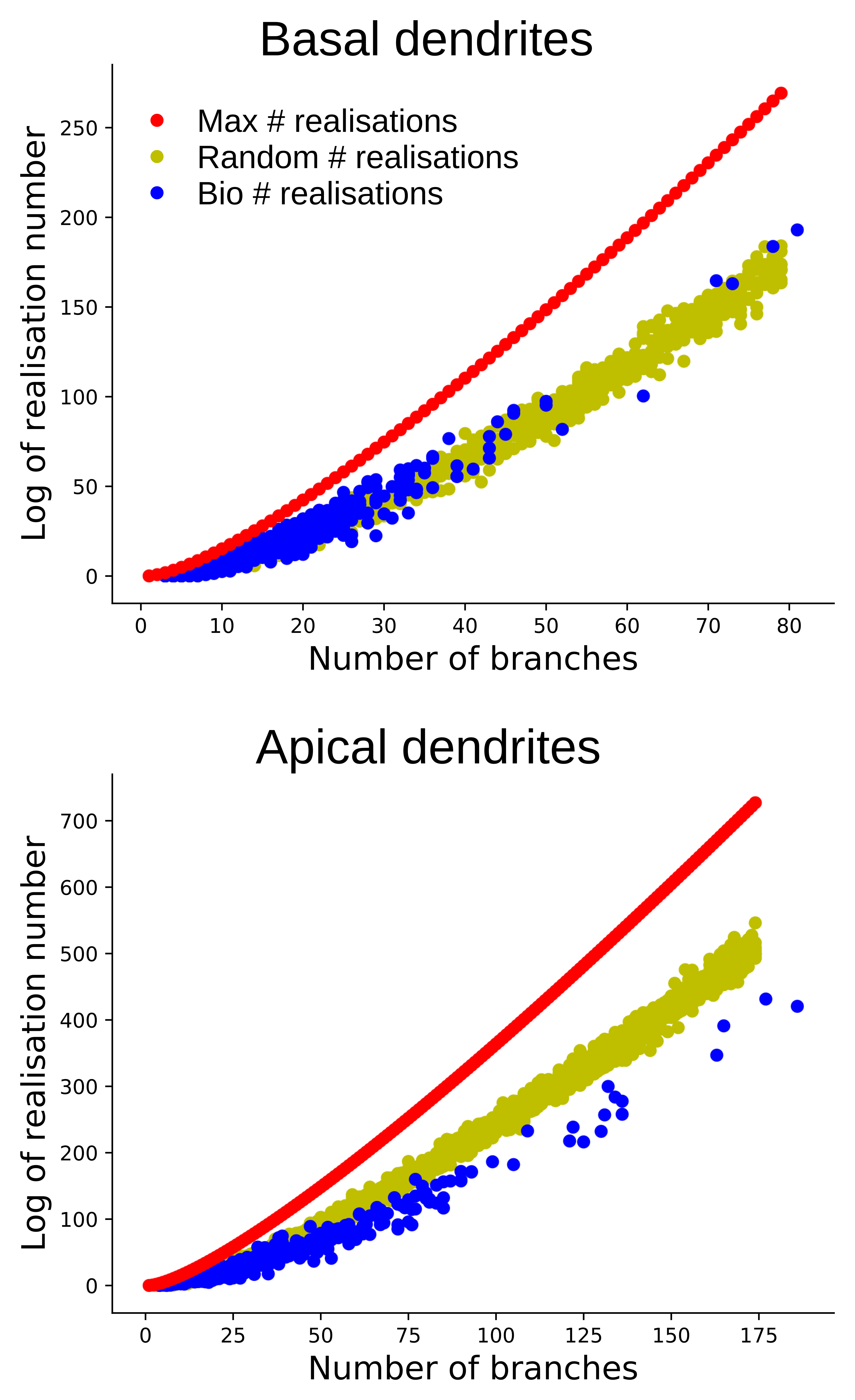}
    \caption{The log of the tree-realization number for barcodes with varying numbers of bars. (A) The log of  tree-realization number for barcodes of basal dendrites (in blue) in comparison with random barcodes (in yellow) and the maximum tree-realization number ($n!$ for $n+1$ bars) (in red). (B) The log of the tree-realization number for barcodes of apical dendrites (in blue) in comparison with random barcodes (in yellow) and the maximum maximum tree-realization number (in red).}
    \label{comp_rand_bio}
\end{figure}
\end{center}

\section{Discussion}


In this paper we presented and analyzed two algorithms that are relevant in topological data analysis: the TMD, which encodes the structure of a geometric tree in a barcode, and the TNS, which generates a geometric tree from a barcode. We proved that for a  good choice of parameter, the TNS is robust with respect to small perturbations of barcodes; an analogous stability result for the TMD was established in \cite{tmd}. 

We observed that ordering  the bars in the persistence barcode according to birth times results in the natural association of a permutation to the barcode, based on death times, giving rise to a meaningful equivalence relation on the set of barcodes. We also introduced a natural, combinatorial equivalence relation on geometric trees. For any barcode, we analyzed the set of combinatorial equivalence classes of those geometric trees whose TMD corresponds that barcode, providing a simple, explicit formula for its cardinality in terms of the permutation associated to the barcode. Cayley graphs of symmetric groups provide a useful visualization of how this cardinality varies as bars in the barcode are transposed.

We illustrated our theoretical results computationally. In addition, we computed the probability for the TNS to generate different combinatorial tree types from a fixed barcode and found it to be a function of the parameter $\lambda$ on which the TNS depends, a result which can be explained only by the stochastic nature of the TNS algorithm. The stochastic nature of the TNS algorithm also leads to variation in the equivalence classes of barcodes associated by the TMD to the trees generated from a fixed barcode by the TNS. In particular, when starting with the TMD of a ``biological'' tree (i.e., arising from a digital neuron reconstruction) including bars with similar birth or death times, we observed an oscillatory behavior between two (or more) different classes states, increasing the variance of the generated trees.

We also initiated an analysis of the distinctive features of biological trees compared to random trees. We discovered that the barcodes associated by the TMD to trees representing neuronal morphologies represent a small fraction of possible equivalence classes of barcodes. It follows that the set of combinatorial types of geometric trees that are biologically realized is also constrained, indicating a biological preference for specific tree structures. There is much yet to discover about which geometric or combinatorial features distinguish biological trees among all geometric trees and why.

In future work  we intend to further investigate the effect of different types of noise on the TNS algorithm. For instance, we have considered only the effect of transposing two bars, but other types of changes are certainly also relevant, such as investigating the effects of switching both births and deaths, as mentioned in Remark \ref{stochastic_several_changes}. On a more neuroscientific note, we intend to continue exploring the distinguishing characteristics of biological trees, with the goal of explaining the structural and functional reasons for the observed geometric and combinatorial constraints.

On the mathematical side, we are currently analyzing the structure on the space of barcodes revealed by the symmetric groups  and determining what information can be extracted from the induced stratification of this space. This structure on the space of barcodes should also provide significant insights into the still somewhat mysterious space of geometric trees, which is of considerable interest to a wide range of mathematicians.


\vspace{6pt} 



\section*{Author contributions}

Conceptualization by L.K., experiments performed by L.K. and A.G. Formal analysis and theorem formulation by L.K., A.G and K.H. Supervision by K.H. All authors wrote that paper and agreed to the published version of the manuscript.

\section*{Funding}

This study was supported by funding to the Blue Brain Project, a research center of the École polytechnique fédérale de Lausanne (EPFL), from the Swiss government’s ETH Board of the Swiss Federal Institutes of Technology. AG and KH gratefully acknowledge the support of Swiss National Science Foundation, grant number CRSII5\_177237.

\section*{Acknowledgments}

The authors would like to thank Justin Curry, Jordan De Sha, and Brendan Mallery for fruitful discussions. We also thank LNMC lab for the neuronal reconstructions used for this analysis, as part of the BBP neuron morphologies dataset.

\section*{Conflicts of interest}

The authors declare no conflict of interest.


\bibliographystyle{plain}
\bibliography{biblio}

\end{document}